\documentclass[12pt, reqno]{amsart}
\usepackage{mathrsfs}
\usepackage{amsfonts}
\usepackage[centertags]{amsmath}
\usepackage{amssymb,array}
\usepackage{amsthm}
\usepackage{graphicx}
\usepackage{caption}
\usepackage{enumerate,enumitem}
\SetLabelAlign{center}{\hfill #1\hfill}
\usepackage[textwidth=16cm, hmarginratio=1:1]{geometry}
\usepackage{tikz-cd, tikz}
\usepackage[colorlinks]{hyperref}
\usepackage{float}
\usepackage{afterpage}

\hypersetup{
bookmarksnumbered,
pdfstartview={FitH},
breaklinks=true,
linkcolor=blue,
urlcolor=blue,
citecolor=blue,
bookmarksdepth=2
}
\usetikzlibrary{arrows.meta}

\theoremstyle{plain}
   \newtheorem{theorem}{Theorem}[section]
   
   \newtheorem{proposition}[theorem]{Proposition}
   
   \newtheorem{lemma}[theorem]{Lemma}

\theoremstyle{definition}
   \newtheorem{definition}[theorem]{Definition}
   
   \newtheorem{example}[theorem]{Example}

   \newtheorem{remark}[theorem]{Remark}

\numberwithin{equation}{section}
\allowdisplaybreaks[2]
\newcommand{\be}{\begin{enumerate}}

    \newcommand{\ene}{\end{enumerate}}

    \newcommand{\ZZ}{\mathbb{Z}}
    
    \newcommand{\RR}{\mathbb{R}}

    \newcommand{\CC}{\mathbb{C}}
    \newcommand{\KK}{\mathbb{K}}

    \newcommand{\Hom}{\operatorname{Hom}}

    \newcommand{\Id}{\operatorname{Id}}

    \newcommand{\bfa}{\mathbf{a}}

    \newcommand{\bfi}{\mathbf{i}}
    
    \newcommand{\bfj}{\mathbf{j}}

    \newcommand{\bt}{\mathbf{t}}

    \newcommand{\cA}{\mathcal{A}}

    \newcommand{\cT}{\mathcal{T}}

    \newcommand{\CM}[1]{\mathcal{M}}

      \newcommand{\gmod}{\operatorname{gmod}}

    \newcommand{\Seteq}[2]{=\left\{{#1}\mid\,{#2}\right\}}
    \newcommand{\Set}[2]{\left\{{#1}\mid\,{#2}\right\}}

\DeclareMathOperator*{\wt}{wt}

\newlength{\mysizetiny}
\newlength{\mysizesmall}
\newlength{\mysize}
\newlength{\mysizelarge}
\begin{document}

\title[Monoidal categories and open Richardson varieties]{Monoidal Categories Associated with Kac--Moody Open Richardson Varieties in Symmetric Type}
\author{Yingjin Bi}
\address{Department of Mathematics, Harbin Engineering University}
\email{yingjinbi@mail.bnu.edu.cn}
\date{} 

\begin{abstract}
We study determinantial modules associated with open Richardson varieties
of symmetric Kac--Moody type. For the seed obtained by the
Bao--Ye--M\'enard mutation procedure, we give an explicit bijection between
its vertices and the positions complementary to the leftmost subexpression.
Each initial variable is a multiplicity-one convolution factor of the
corresponding determinantial module, and these determinantial modules are
cluster monomials in that seed. We deduce that the quantum cluster algebra
with non-invertible frozen variables embeds in the Grothendieck ring of
$\mathscr C_{w,v}$, with cluster monomials represented by real simple modules
up to normalization. In finite $ADE$ type, we identify the seed with
Leclerc's seed and obtain a monoidal categorification after localization
at the frozen variables.
\end{abstract}

\raggedbottom
\afterpage{\flushbottom}

\maketitle
\par\kern-1\baselineskip
\tableofcontents
\par\vspace{-2\baselineskip}

\section{Introduction}

Let $G$ be a symmetric Kac--Moody group over $\CC$, with opposite Borel subgroups $B^+$ and $B^-$. Denote by $\mathcal{B}=G/B^+$ the flag variety. For Weyl group elements $v,w\in W$ with $v\leq w$ in the Bruhat order, the intersection
\[
\mathring{\mathcal{B}}_{v,w}:=\mathring{\mathcal{B}}_w\cap \mathring{\mathcal{B}}^{v}
\]
of the Schubert cell $\mathring{\mathcal{B}}_w=B^+wB^+/B^+$ and the opposite Schubert cell $\mathring{\mathcal{B}}^{v}=B^-vB^+/B^+$ is called an \emph{open Richardson variety}.

A central problem is to describe the cluster structures on
$\CC[\mathring{\mathcal B}_{v,w}]$ and their representation-theoretic
realizations. In finite $ADE$ type, Leclerc \cite{leclerc2016cluster}
constructed a cluster subalgebra using Frobenius subcategories of modules
over a preprojective algebra; the variables of a distinguished initial
cluster are irreducible factors of generalized minors. Geometric cluster
structures on Richardson and braid varieties were subsequently developed
in \cite{casals2025cluster,galashin2025braid,serhiyenko2024leclerc}.
M\'enard \cite{menard2022cluster} gave an algorithm for computing initial
seeds, and Bao--Ye \cite{bao2025upper} established an upper cluster
structure on open Richardson varieties in all symmetrizable Kac--Moody
types.

On the monoidal side, Kang--Kashiwara--Kim--Oh
\cite{kang2018monoidal} categorified quantum unipotent coordinate rings by
categories $\mathscr C_w$ of modules over symmetric quiver Hecke algebras.
Kashiwara--Kim--Oh--Park
\cite{kashiwara2018monoidal,kashiwara2023localizations} introduced the
subcategories $\mathscr C_{w,v}$ and their localizations. The issue
addressed here is the explicit identification of the variables in the
Bao--Ye--M\'enard seed within this monoidal framework, rather than the
existence of categorical models for cluster algebras in general.

An earlier finite-Dynkin treatment by the author
\cite{bi2024richardson} approached this problem through Lusztig
parametrizations. The present paper develops a determinantial-factorization
approach in symmetric Kac--Moody type. Its main additional content is the
explicit multiplicity-one correspondence between factors and seed
variables, together with the comparison with M\'enard's seed; the
combinatorial reindexing alone is not the categorification argument.
For the finite-type localization statement, we give a proof using
ambient Laurent positivity and the classical Richardson realization.

Our main construction follows the mutation sequence attached to a
leftmost subexpression and identifies its variables as multiplicity-one
factors of determinantial modules. Conversely, the relevant determinantial
modules are products of variables from this single seed. This description
connects the combinatorial mutation algorithm with Leclerc's
factorization description.

To state the result, fix a reduced expression $\overline{w}=(i_1,\dots,i_r)$ of $w$, and let $\overline{v}=(i_{p_1},\dots,i_{p_m})$ be the leftmost subexpression corresponding to $v$. Let $\mathbf{s}(\overline{v},\overline{w})$ be the associated seed, with vertex set $J$.

\begin{theorem}[{Theorems~\ref{thm:dividminors} and~\ref{thm:menard}}]
There is a monoidal realization $(M_k)_{k\in J}$ of the initial
variables of $\mathbf{s}(\overline v,\overline w)$ in $\mathscr C_{w,v}$
and a canonical bijection
\[
\Phi: J \longrightarrow \{1,\dots,r\}\setminus\{p_1,\dots,p_m\}
\]
such that $M_k$ occurs as a multiplicity-one factor of the determinantial module
\[
M\!\binom{w_{\le \Phi(k)}\varpi_{i_{\Phi(k)}}}{\overline v_{\Phi(k)}\varpi_{i_{\Phi(k)}}}.
\]
Moreover, for each $l\notin \{p_1,\dots,p_m\}$, the determinantial module
\[
M\!\binom{w_{\le l}\varpi_{i_l}}{\overline v_l\varpi_{i_l}}
\]
is, up to a grading shift, a cluster monomial in this seed. In finite
$ADE$ type, M\'enard's seed and Leclerc's seed coincide up to relabeling,
with both written in the same global quiver convention.
\end{theorem}

We distinguish the non-localized statement in symmetric Kac--Moody type
from the localized categorification statement in finite type.

\begin{theorem}[Theorem~\ref{thm:clustercat}, Theorem~\ref{thm:categorification}]
Let $v\le w$ be Weyl group elements of symmetric Kac--Moody type. Then
\[
\overline{\mathcal A}_q(\mathbf s(\overline v,\overline w))
\subset K(\mathscr C_{w,v}),
\]
where the frozen variables on the left are not inverted. Every cluster
monomial is represented by a real simple module, up to the standard
normalization. In finite $ADE$ type,
\[
K(\widetilde{\mathscr C}_{w,v})
\simeq\mathcal A_q(\mathbf s(\overline v,\overline w)),
\]
where the frozen variables are inverted. Thus
$\widetilde{\mathscr C}_{w,v}$ gives a monoidal categorification of this
quantum cluster algebra.
\end{theorem}

\medskip

\noindent
After the first version of this work, Kashiwara--Kim--Oh--Park
\cite{kashiwara2026seeds} constructed monoidal seeds in $\mathscr C_{w,v}$
using reflection functors and the operators $\mathcal K_i$. They identified
the same family of prime determinantial factors and proved a
Laurent-family property. Their paper explicitly relates these results to
the present work. 
For the final upper inclusion, we use the positive Laurent
expansions of \cite[Lemma~2.3]{kashiwara2019laurent}: the classical
Richardson realization excludes all Laurent monomials involving deleted
variables, and positivity lifts this support restriction to the quantum
algebra.

\medskip

\noindent
\textbf{Structure of the paper.}
In Section~\ref{sec:cluster}, we recall basic material on cluster algebras. In Section~\ref{sec:kacmoody}, we review open Richardson varieties and their cluster structures. In Section~\ref{sec:quiverhecke}, we study determinantial modules. Finally, in Section~\ref{sec:monoidalcat}, we prove the main results.

\medskip

\noindent
\textbf{Acknowledgments.}
The author would like to express his sincere gratitude to Masaki Kashiwara and Ryo Fujita for many helpful discussions and valuable suggestions.

\section{Cluster algebras}\label{sec:cluster}
\subsection{Definition of cluster algebras}
In this section, we will recall the notion of cluster algebras and quantum cluster algebras.
\subsubsection{Cluster algebras}
Let $I$ be a finite set together with a partition
\[
I = I_{\mathrm{uf}} \sqcup I_{\mathrm{fr}}.
\]
Let $B=(b_{ij})$ be an $I\times I_{\mathrm{uf}}$ integer matrix such that the principal part
\[
B_{I_{\mathrm{uf}}\times I_{\mathrm{uf}}}
\]
is skew-symmetric. The elements of $I_{\mathrm{uf}}$ are called \emph{mutable vertices}, and those of $I_{\mathrm{fr}}$ are called \emph{frozen vertices}. We associate with $B$ the ice quiver $Q_B$: a positive entry $b_{ij}$ gives $b_{ij}$ arrows from $i$ to $j$, and a negative entry gives $-b_{ij}$ arrows from $j$ to $i$. Each pair is counted only once, and arrows between frozen vertices are omitted.

Consider the Laurent polynomial ring
\[
\CC[x_i^{\pm1}\mid i\in I],
\]
and let $\CC(x_i\mid i\in I)$ be its field of fractions. For simplicity, we write $\CC[x_i^{\pm1}]$ and $\CC(x_i)$, respectively.

A \emph{seed} is a triple
\[
\bt=\bigl((x_i)_{i\in I},\, B,\, I_{\mathrm{fr}}\bigr).
\]
For $\bfa=(a_i)_{i\in I}\in \ZZ^I$, we set
\[
x^\bfa:=\prod_{i\in I}x_i^{a_i}.
\]

\begin{definition}[Mutation]
Let $k\in I_{\mathrm{uf}}$. The mutation $\mu_k$ at $k$ is defined as follows.
\begin{enumerate}
    \item The mutated exchange matrix $\mu_k(B)=(\mu_k(B)_{ij})$ is given by
    \[
    \mu_k(B)_{ij}=
    \begin{cases}
    -b_{ij}, & \text{if } i=k \text{ or } j=k,\\[1mm]
    b_{ij}+[b_{ik}]_+[b_{kj}]_+-[-b_{ik}]_+[-b_{kj}]_+, & \text{otherwise},
    \end{cases}
    \]
    where $[a]_+:=\max\{a,0\}$.

    \item The mutated cluster variables are
    \[
    \mu_k(x_i)=
    \begin{cases}
    x^{\bfa'}+x^{\bfa''}, & \text{if } i=k,\\
    x_i, & \text{if } i\neq k,
    \end{cases}
    \]
    where $\bfa'=(a_i')_{i\in I}$ and $\bfa''=(a_i'')_{i\in I}$ are defined by
    \begin{equation}\label{eq:bfa}
    a_i'=
    \begin{cases}
    -1, & \text{if } i=k,\\
    \max(0,b_{ik}), & \text{if } i\neq k,
    \end{cases}
    \qquad
    a_i''=
    \begin{cases}
    -1, & \text{if } i=k,\\
    \max(0,-b_{ik}), & \text{if } i\neq k.
    \end{cases}
    \end{equation}
\end{enumerate}
The seed
\[
\mu_k(\bt):=\bigl((\mu_k(x_i))_{i\in I},\,\mu_k(B),\,I_{\mathrm{fr}}\bigr)
\]
is called the mutation of $\bt$ at $k$.
\end{definition}

For a seed $\bt'=\bigl((x_i')_{i\in I},\,B',\,I_{\mathrm{fr}}\bigr)$, the elements $x_i'$ are called the \emph{cluster variables}, and $B'$ is called the \emph{exchange matrix}. If $i\in I_{\mathrm{fr}}$, then $x_i'$ is called a \emph{frozen variable}.

Let $T$ be the set of seeds obtained from $\bt$ by finite sequences of mutations.

\begin{definition}
The \emph{cluster algebra} $\mathcal{A}(\bt)$ associated with the seed $\bt$ is the $\CC$-subalgebra of $\CC(x_i)$ generated by all cluster variables appearing in all seeds $\bt'\in T$. In our convention, the frozen variables are not assumed to be invertible.

The corresponding \emph{upper cluster algebra} is defined by
\[
U(\bt):=\bigcap_{\bt'\in T}\CC[x_{\bt',i}^{\pm1}],
\]
where $x_{\bt',i}$ denotes the $i$th cluster variable of the seed $\bt'$.
\end{definition}

By the Laurent phenomenon, one has
\[
\mathcal{A}(\bt)\subset U(\bt).
\]

\subsubsection{Quantum cluster algebras}

Let $K$ be a finite vertex set, let $K^{\mathrm{ex}}\subseteq K$ be its mutable subset, and let $B$ be an integer $K\times K^{\mathrm{ex}}$ exchange matrix. Let $\Lambda=(\lambda_{ij})$ be a skew-symmetric integer $K\times K$ matrix. We use the symmetric-type normalization in which the pair $(\Lambda,B)$ is \emph{compatible} if
\[
\sum_{k\in K}\lambda_{ik}b_{kj}=2\delta_{ij}
\qquad
\text{for all } i\in K,\ j\in K^{\mathrm{ex}}.
\]

Given such a skew-symmetric matrix $\Lambda$, we define the \emph{quantum torus} $\mathcal{T}_\Lambda$ to be the $\mathbb{K}$-algebra generated by $X_i^{\pm1}$ $(i\in K)$, where $\mathbb{K}=\mathbb{Z}[q^{\pm1/2}]$, subject to the relations
\[
X_iX_j=q^{\lambda_{ij}}X_jX_i,
\qquad
X_iX_i^{-1}=X_i^{-1}X_i=1.
\]

After ordering $K=\{1,\ldots,r\}$, for any vector $\mathbf{a}=(a_1,\dots,a_r)\in \ZZ^K$, define the normalized monomial
\[
X^{\mathbf{a}}
=
q^{\frac12\sum_{i>j}a_ia_j\lambda_{ij}}X_1^{a_1}\cdots X_r^{a_r}.
\]

A \emph{quantum seed} is a tuple
\[
\mathbf{t}:=\bigl((X_i)_{i\in K},\,\Lambda,\,B,\,K^{\mathrm{ex}}\bigr),
\]
where $(\Lambda,B)$ is a compatible pair.

For $k\in K^{\mathrm{ex}}$, the mutation at $k$ is defined as follows. The mutated matrix $\mu_k(\Lambda)$ has entries
\begin{equation}\label{eq_mutations}
\mu_k(\Lambda)_{ij}=
\begin{cases}
-\lambda_{kj}+\sum_{l\in K}\max\{0,-b_{lk}\}\lambda_{lj}, & \text{if } i=k,\ j\neq k,\\[1mm]
-\lambda_{ik}+\sum_{l\in K}\max\{0,-b_{lk}\}\lambda_{il}, & \text{if } i\neq k,\ j=k,\\[1mm]
\lambda_{ij}, & \text{otherwise}.
\end{cases}
\end{equation}
The mutated cluster variables are
\[
\mu_k(X_i)=
\begin{cases}
X_i, & \text{if } i\neq k,\\
X^{\mathbf{a}'}+X^{\mathbf{a}''}, & \text{if } i=k.
\end{cases}
\]
where $\bfa'$ and $\bfa''$ are given by \eqref{eq:bfa}, with the vertex set $I$ replaced by $K$.

It is well known that $(\mu_k(\Lambda),\mu_k(B))$ is again a compatible pair. Thus one obtains a new quantum seed
\[
\mu_k(\mathbf{t})
:=
\bigl((\mu_k(X_i))_{i\in K},\,\mu_k(\Lambda),\,\mu_k(B),\,K^{\mathrm{ex}}\bigr).
\]

\begin{definition}
For a quantum seed $\mathbf{t}$, the \emph{quantum cluster algebra} $\mathcal{A}_q(\mathbf{t})$ is the $\mathbb{K}$-subalgebra generated by all cluster variables $X_i(\mathbf{t}')$ for all seeds $\mathbf{t}'\in T$, together with the inverses of the frozen variables. The \emph{upper quantum cluster algebra} is defined by
\[
U_q(\mathbf{t})=\bigcap_{\bt'\in T}\mathcal{T}(\bt'),
\]
where $\mathcal{T}(\bt')$ denotes the quantum torus associated with the seed $\bt'$. We denote by $\overline{\cA}_q(\bt)$ the $\KK$-subalgebra of $\cT(\bt)$ generated by all cluster variables $X_i(\bt')$ for all seeds $\bt'\in T$. 
All quantum tori in this intersection are viewed in the common skew field of fractions over $\mathbb Q(q^{1/2})$. Thus $\overline{\mathcal A}_q$ is the convention without inverted frozen variables, whereas $\mathcal A_q$ includes those inverses.
\end{definition}

\subsection{Cluster algebras associated with words}\label{sec:wordseed}

In this subsection, we recall the cluster algebra structures associated with words in \(I\).

\subsubsection{Cluster algebras arising from words}
Let $C$ be a symmetric $I\times I$ Cartan matrix. A \emph{word} in $I$ is a sequence
\[
\bfi=(i_1,\dots,i_r),
\]
where $i_k\in I$ for all $1\le k\le r$. We write $[r]:=\{1,\dots,r\}$ for $r\geq1$ and $[0]:=\varnothing$. The color set $I$ is kept distinct from the vertex set $[r]$.

For $k\in [r]$ and $j\in I$, we say that $k$ has \emph{color} $j$ if $i_k=j$. For each $j\in I$, let $n_j$ denote the number of occurrences of the color $j$ in the word $\bfi$.

For $k\in [r]$, define
\[
k^+:=\min\{p>k\mid i_p=i_k\},
\qquad
k^-:=\max\{p<k\mid i_p=i_k\},
\]
with the convention that $k^+=+\infty$ and $k^-=-\infty$ when the corresponding sets are empty.

It is often convenient to label the $l$th occurrence of the color $j$ by $(j,l)$. Thus, if $k=(j,l)$, then $k^+=(j,l+1)$ and $k^-=(j,l-1)$. We also set $(j,0)=-\infty$ and $(j,n_j+1)=+\infty$.

For $j\neq i_k$, define
\[
k(j)^-:=\max\{p<k\mid i_p=j\},
\qquad
k(j)^+:=\min\{p>k\mid i_p=j\},
\]
again allowing the values $\pm\infty$.

Let $\bfi$ be a word of length $r\geq0$. Set
\[
J:=[r],
\qquad
J_{\mathrm{fr}}:=\{k\in J\mid k^+=+\infty\},
\qquad J_{\mathrm{uf}}:=J\setminus J_{\mathrm{fr}}.
\]
We define the exchange matrix $B_{\bfi}=(b_{pq})_{p\in J,\,q\in J_{\mathrm{uf}}}$ by
\[
b_{pq}=
\begin{cases}
-1, & \text{if } p=q^-,\\
1, & \text{if } q=p^-,\\
c_{i_p,i_q}, & \text{if } q<p<q^+<p^+,\\
-c_{i_p,i_q}, & \text{if } p<q<p^+<q^+,\\
0, & \text{otherwise}.
\end{cases}
\]

We denote by
\[
\mathbf{s}(\bfi):=\bigl((x_k)_{k\in J},\,B_{\bfi},\,J_{\mathrm{fr}}\bigr)
\]
the seed associated with the word $\bfi$. The corresponding cluster algebra and upper cluster algebra are denoted by
\[
\cA(\bfi):=\cA(\mathbf{s}(\bfi)),
\qquad
U(\bfi):=U(\mathbf{s}(\bfi)).
\]

For a word $\bfi=(i_1,\dots,i_r)$, define
\[
w_{\le k}=s_{i_1}\cdots s_{i_k}
\qquad
(k\in [r]),\qquad w_{\leq0}:=e.
\]
Use a symmetric Kac--Moody realization as in Section~\ref{sec:kacmoody}, with $(\alpha_i,\alpha_i)=2$ and $(\varpi_i,\alpha_j)=\delta_{ij}$. Let $\Lambda_{\bfi}=(\lambda_{kl})$ be the skew-symmetric matrix determined by
\[
\lambda_{kl}=-\lambda_{lk}=
(w_{\le k}\varpi_{i_k}+\varpi_{i_k},\,w_{\le l}\varpi_{i_l}-\varpi_{i_l})
\qquad
\text{for } k>l.
\]
It is well known that $(\Lambda_{\bfi},B_{\bfi})$ is a compatible pair. We denote by
\[
\mathbf{s}(\bfi):=\bigl((x_k)_{k\in J},\,\Lambda_{\bfi},\,B_{\bfi},\,J_{\mathrm{uf}}\bigr)
\]
the associated quantum seed, and write
\[
\cA_q(\bfi):=\cA_q(\mathbf{s}(\bfi)).
\]
The same notation is used for the classical and quantum seeds when the meaning is clear. We use the quiver and quantum commutation conventions of \cite{kang2018monoidal}, with $b_{pq}>0$ represented by arrows $p\to q$. Geometric source seeds are transported to this convention by reversing all arrows when necessary. The simultaneous change $(\Lambda,B)\mapsto(-\Lambda,-B)$ preserves compatibility, and $\mu_k(-B)=-\mu_k(B)$; moreover, reversing arrows commutes with freezing and deletion and leaves the commutative exchange variables unchanged. Thus all seed comparisons below use this single convention. For the empty word, all matrices and vertex sets are empty, and the associated algebras are their coefficient rings.
\subsubsection{Cluster algebras induced by subwords}
Now let $\bfj$ be a subword of $\bfi$. We write
\[
\bfj=(i_{p_1},\dots,i_{p_m})=(j_1,\dots,j_m),
\]
where $p_k<p_{k+1}$ whenever $1\leq k<m$. Write $\bfj_k:=(i_{p_1},\ldots,i_{p_k})$ and let $\bfj_0$ be the empty word.

For $k\in [m]$, let $i:=i_{p_k}$ and define
\[
a_k:=\#\{s\in [k]\mid i_{p_s}=i\},
\qquad
b_k:=\#\{t\in [p_k]\mid i_t=i \text{ and } t\notin \{p_1,\dots,p_k\}\},
\]
and set
\[
d_k:=a_k+b_k.
\]
Then
\[
p_k=(i,a_k+b_k)=(i,d_k).
\]
We also define
\[
\alpha(j,k):=\#\{t\in [k]\mid i_{p_t}=j\}
\qquad(0\leq k\leq m).
\]

\begin{definition}\label{def:subwordmutation}
Let $\bfi$ be a word and let $\bfj$ be a subword of $\bfi$. For each $k\in [m]$, define a sequence of mutations $\widetilde{\mu}_k$ by
\[
\widetilde{\mu}_k=
\begin{cases}
\mu_{(i,n_i-a_k)}\circ \cdots \circ \mu_{(i,b_k+1)}, & \text{if } a_k+b_k<n_i,\\
\Id, & \text{otherwise},
\end{cases}
\]
where $i=i_{p_k}$.

Set
\[
\widetilde{\mathbf{s}}(\bfj_k,\bfi)
:=
\widetilde{\mu}_k\circ \widetilde{\mu}_{k-1}\circ \cdots \circ \widetilde{\mu}_1\bigl(\mathbf{s}(\bfi)\bigr).
\]
The seed $\mathbf{s}(\bfj_k,\bfi)$ is obtained from $\widetilde{\mathbf{s}}(\bfj_k,\bfi)$ by deleting the vertices
\[
D_k:=\{(j,n_j-\ell+1)\mid j\in I,\ 1\leq\ell\leq\alpha(j,k)\}
\]
and freezing every remaining mutable vertex adjacent to a vertex of $D_k$ in the quiver of $\widetilde{\mathbf{s}}(\bfj_k,\bfi)$. Already frozen vertices remain frozen, and arrows between frozen vertices are discarded. For $k=0$, we set
\[
\widetilde{\mathbf{s}}(\bfj_0,\bfi)
=\mathbf{s}(\bfj_0,\bfi):=\mathbf{s}(\bfi).
\]

Denote the matrices of $\widetilde{\mathbf{s}}(\bfj_k,\bfi)$ by $\widetilde B$ and $\widetilde\Lambda$. Restrict $\widetilde\Lambda$ to the remaining vertices and $\widetilde B$ to the remaining rows and mutable columns. This yields a compatible pair $(\Lambda_{\bfj_k,\bfi},B_{\bfj_k,\bfi})$: for every retained mutable vertex $u$, one has $\widetilde b_{du}=0$ for $d\in D_k$, and hence
\[
\sum_{z\notin D_k}\widetilde\lambda_{pz}\widetilde b_{zu}
=\sum_z\widetilde\lambda_{pz}\widetilde b_{zu}
=2\delta_{pu}.
\]
This definition performs all mutations in the full seed before deleting vertices; no intermediate deletion is implicit.

In particular, we write $\mathbf{s}(\bfj,\bfi):=\mathbf{s}(\bfj_m,\bfi)$ and define
\[
\cA(\bfj,\bfi):=\cA(\mathbf{s}(\bfj,\bfi)),
\qquad
\cA_q(\bfj,\bfi):=\cA_q(\mathbf{s}(\bfj,\bfi)),
\]
\[
U(\bfj,\bfi):=U(\mathbf{s}(\bfj,\bfi)),
\qquad
U_q(\bfj,\bfi):=U_q(\mathbf{s}(\bfj,\bfi)).
\]
\end{definition}

\begin{example}\label{exm:words}
Consider type $A_4$ and the reduced expression
\[
\overline{w}=(1234123121)
\]
of $w_0$. Let
\[
v=s_2s_3s_1s_2s_1<w_0.
\]
The corresponding subexpression $\overline{v}$ of $v$ inside $\overline{w}$ is
\[
(1,\underline{2},\underline{3},4,\underline{1},\underline{2},3,\underline{1},2,1),
\]
so that
\[
(p_1,\dots,p_5)=(2,3,5,6,8).
\]
Moreover,
\[
a_1=1,\ a_2=1,\ a_3=1,\ a_4=2,\ a_5=2,
\qquad
b_1=b_2=0,\ b_3=1,\ b_4=0,\ b_5=1.
\]
It follows that
\[
\widetilde{\mu}_1=\mu_{(2,2)}\mu_{(2,1)}=\mu_6\mu_2,\qquad
\widetilde{\mu}_2=\mu_{(3,1)}=\mu_3,
\]
\[
\widetilde{\mu}_3=\mu_{(1,3)}\mu_{(1,2)}=\mu_8\mu_5,\qquad
\widetilde{\mu}_4=\mu_{(2,1)}=\mu_2,\qquad
\widetilde{\mu}_5=\mu_{(1,2)}=\mu_5.
\]
\end{example}
\subsection{Leftmost subexpressions}\label{sec:leftmost}
We introduce the notation used to describe factors of determinantial modules in $\mathscr C_{w,v}$.

Let $\overline w=(i_1,\ldots,i_r)$ be a reduced expression of $w$, and let $v\leq w$. Among the position sequences of reduced subexpressions for $v$, choose the lexicographically smallest one,
\[
p_1<\cdots<p_m,\qquad
\overline v=(i_{p_1},\ldots,i_{p_m}),\qquad m=\ell(v).
\]
We call $\overline v$ the \emph{leftmost subexpression}. All lexicographic comparisons refer to position sequences, not to color labels. For $0\leq k\leq r$, set
\[
w_k:=s_{i_1}\cdots s_{i_k},\qquad
\overline v_k:=\prod_{p_t\leq k}s_{i_{p_t}},
\]
with products taken in increasing order and empty products equal to $e$. Thus $\overline v_k\leq w_k$.

Following \cite{leclerc2016cluster}, define $v'_0=e$ and
\[
v'_k=
\begin{cases}
v'_{k-1}s_{i_k} & \text{if } s_{i_k}v_{k-1}^{'-1}v < v_{k-1}^{' -1}v,\\
v'_{k-1} & \text{otherwise}.
\end{cases}
\]

\begin{lemma}\label{lem:vk}
Let $\overline{w}=(i_1\cdots i_r)$ be a reduced expression of $w$, and let $v\le w$.
Then, for $0\leq k\leq r$, one has
\[
\overline{v}_k = v'_k,\qquad v'_r=v.
\]
\end{lemma}

\begin{proof}
The assertion is clear for $k=0$. Suppose it holds at $k-1$, and put
$u:=\overline v_{k-1}^{-1}v$. If $k$ is selected, the remaining reduced expression of $u$ starts with $i_k$, so $s_{i_k}u<u$ and both recursions select $k$.

If $k$ is not selected, the remaining selected positions lie in $(i_{k+1},\ldots,i_r)$. Were $s_{i_k}u<u$, the subword property of Bruhat order would give a reduced subexpression $\gamma$ of this suffix for $s_{i_k}u$. The previously selected positions, followed by $k$ and $\gamma$, would then give a reduced subexpression of $v$ lexicographically smaller than $\overline v$. This is impossible. Thus $s_{i_k}u>u$ and both recursions skip $k$. Induction proves the claim, including $v'_r=\overline v_r=v$.
\end{proof}

For $0\leq l\leq m$, set
\[
\overline v^{\,l}:=(i_{p_1},\ldots,i_{p_l}),\qquad
v^l:=s_{i_{p_1}}\cdots s_{i_{p_l}},
\qquad \overline v^{\,0}=\varnothing,\quad v^0=e.
\]
Thus $\overline v^{\,l}$ is a word indexed by the number of selected
positions, whereas $\overline v_k$ is the group element obtained by
truncating at the absolute word position $k$. In particular,
$\mathbf s(\overline v^{\,l},\overline w)$ denotes the seed at stage $l$;
the superscript here must not be replaced by a subscript.
For $0\leq k\leq r$, define
\[
\overline v_k^l:=\prod_{\substack{t\leq l\\p_t\leq k}}s_{i_{p_t}}.
\]
\begin{lemma}\label{lem:leftmost-prefix}
The positions $p_1,\ldots,p_l$ form the leftmost subexpression of $v^l$ in $\overline w$. In particular, for $l\geq1$,
\[
v^l\not\leq w_{p_l-1}.
\]
\end{lemma}
\begin{proof}
Apply the greedy recursion of Lemma~\ref{lem:vk} with target $v^l$. After the first $t<l$ selected positions, its residual target is $a:=(v^t)^{-1}v^l$, whereas the residual target for $v$ is $ac:=(v^t)^{-1}v$, with $\ell(ac)=\ell(a)+\ell(c)$. Every left descent of $a$ is a left descent of $ac$: if $s_i a<a$, then $\ell(s_iac)\leq\ell(ac)-1$. Thus a position skipped for $v$ is also skipped for $v^l$ until $p_l$. Each position $p_{t+1}$, for $t<l$, is selected because the displayed reduced suffix for $a$ starts with $i_{p_{t+1}}$. After $p_l$ the residual target is $e$, so the recursion stops.

If $v^l\leq w_{p_l-1}$, Lemma~\ref{lem:vk} applied to that shorter reduced word would recover $v^l$ before position $p_l$. This contradicts the same greedy recursion just described.
\end{proof}

For $l\geq1$, one has $\overline v_k^{\,l-1}=\overline v_k^{\,l}$ if $k<p_l$, and
\[
\overline{v}^{\,l-1}_k s_{i_{p_l}} = \overline{v}^{\,l}_k
\]
for all $k \ge p_l$.

For each $j\in I$, we define
\begin{equation}
I_j:=\{(j,k)\mid 1\leq k\leq n_j\}\subseteq[r],
\end{equation}
the subset of $\overline{w}$ consisting of vertices of color $j$.

Here a colored pair is identified with its position in the original word. For $0\leq l\leq m$, set
\[
I_j^l:=I_j^{v^l}:=I_j\cap\{p_1,\ldots,p_l\},
\qquad I_j^v:=I_j^m.
\]
Fixing $j,l$, write $a:=\alpha(j,l)=|I_j^l|$ and enumerate
\[
I_j^l=\{(j,t_1),\ldots,(j,t_a)\},\qquad
t_1<\cdots<t_a.
\]
The numbers $t_s$ are occurrence numbers for this fixed color, not indices in the global sequence $(p_1,\ldots,p_l)$. Set $t_0:=0$ and $t_{a+1}:=n_j+1$, and define, for $0\leq s\leq a$,
\[
I_{j,s}^l:=\{(j,k)\in I_j\mid t_s\leq k<t_{s+1}\}.
\]
Then $I_j=\bigsqcup_{s=0}^a I_{j,s}^l$. In particular, if $a=0$, this convention gives $I_{j,0}^l=I_j$, including when $I_j=\varnothing$. The symbol $(j,t_0)=(j,0)$ is only a boundary symbol and is not a vertex.

If $i=i_{p_l}$, then $\alpha(i,l)=\alpha(i,l-1)+1$ and $I_i^l=I_i^{l-1}\sqcup\{p_l\}$. Thus the last interval at stage $l-1$ splits into two intervals:
\begin{equation}\label{eq:Iisl}
I^{l-1}_{i,s}=
\begin{cases}
I^l_{i,s}, & s<\alpha(i,l-1),\\[4pt]
I^l_{i,\alpha(i,l-1)}\cup I^l_{i,\alpha(i,l)}, & s=\alpha(i,l-1).
\end{cases}
\end{equation}
For $j\neq i$, the partitions are unchanged:
\[
I^{l-1}_{j,s}=I^l_{j,s}
\quad
\text{for }0\leq s\leq\alpha(j,l).
\]
For $T\subseteq I_j$ and $s\geq0$, define
\[
T[-s]:=\{(j,k)\in I_j\mid (j,k+s)\in T\}.
\]
For a fixed color $i$ and $0\leq s\leq\alpha(i,l)$, the endpoint conventions give
\[
\left(I^l_{i,s}\setminus\{(i,t_s)\}\right)[-s]
=\{(i,k)\in I_i\mid t_s-s+1\leq k\leq t_{s+1}-s-1\}.
\]
These intervals are disjoint and consecutive after empty intervals are omitted. Therefore
\begin{equation}\label{eq:partitionIi}
I_{i,\leq l}:=\{(i,k)\in I_i \mid k \leq n_i - \alpha(i,l)\}
=
\bigsqcup_{s=0}^{\alpha(i,l)}
\left(I^l_{i,s}\setminus\{(i,t_s)\}\right)[-s].
\end{equation}
Equivalently, enumerate the unselected occurrence numbers as
\[
\{u_{i,1}<\cdots<u_{i,n_i-\alpha(i,l)}\}
=[n_i]\setminus\{t_1,\ldots,t_{\alpha(i,l)}\}.
\]
The order-preserving bijection $\Phi_i^l(i,k):=(i,u_{i,k})$ is given explicitly by
\begin{equation}\label{eq:Phiil}
\begin{split}
         \Phi_{i}^l:I_{i,\leq l}&\to I_i\setminus I_i^l\\
            (i,k)&\mapsto (i,k+s),  \text{ if }(i,k)\in \left(I^l_{i,s}\setminus\{(i,t_s)\}\right)[-s]
\end{split}
\end{equation}
Indeed, precisely $s$ selected occurrences precede an unselected occurrence in $I_{i,s}^l$, so its rank among unselected occurrences is reduced by $s$. This also proves bijectivity. If $\alpha(i,l)=0$, then $\Phi_i^l$ is the identity; if all occurrences have been selected, it is the unique map between empty sets.

Gluing the maps for all colors yields the bijection
\[
\Phi^l :\, \bigsqcup_{i\in I} I_{i,\leq l} \to [r]\setminus\{p_1,\dots,p_l\}
\]
All comparisons between its image vertices use their positions in the original word $\overline w$.

Finally, suppose $i=i_{p_l}$. Since $t_{\alpha(i,l)}=d_l=b_l+\alpha(i,l)$, the newly affected interval is
\[
N_l:=\left(I^l_{i,\alpha(i,l)}\setminus\{p_l\}\right)[-\alpha(i,l)]
=\{(i,k)\in I_i\mid b_l+1\leq k\leq n_i-\alpha(i,l)\}.
\]
On $I_{j,\leq l}$, the maps $\Phi_j^{l-1}$ and $\Phi_j^l$ coincide outside $N_l$. On $N_l$ they satisfy
\[
\Phi_i^{l-1}(i,k)=(i,k+\alpha(i,l)-1),\qquad
\Phi_i^l(i,k)=(i,k+\alpha(i,l)).
\]

\begin{example}
Following Example \ref{exm:words}, let us consider the reduced expression
\[
\overline{w}=(1234123121)
\]
of $w_0$. Let 
\[
v=s_2s_3s_1s_2s_1<w_0 .
\]
The leftmost subexpression $\overline{v}$ of $v$ in $\overline{w}$ is
\[
(1,\underline{2},\underline{3},4,\underline{1},\underline{2},3,\underline{1},2,1),
\]
and hence
\[
(p_1,\dots,p_5)=(2,3,5,6,8).
\]

\begin{table}[h]
\centering
\begin{tabular}{c c c c c c c c c c c}
\hline
$/$ & $1$ & $2$ & $3$ & $4$ & $5$ & $6$ & $7$ & $8$ & $9$ & $10$ \\
\hline
$w$ & $w_1$ & $w_2$ & $w_3$ & $w_4$ & $w_5$ & $w_6$ & $w_7$ & $w_8$ & $w_9$ & $w_{10}$ \\
$\overline{v}$ & $e$ & $s_2$ & $s_2s_3$ & $s_2s_3$ & $s_2s_3s_1$ & $s_2s_3s_1s_2$ & $s_2s_3s_1s_2$ & $v$ & $v$ & $v$ \\
$\overline{v}^4$ & $e$ & $s_2$ & $s_2s_3$ & $s_2s_3$ & $s_2s_3s_1$ & $v^4$ & $v^4$ & $v^4$ & $v^4$ & $v^4$ \\
$\overline{v}^3$ & $e$ & $s_2$ & $s_2s_3$ & $s_2s_3$ & $v^3$ & $v^3$ & $v^3$ & $v^3$ & $v^3$ & $v^3$ \\
$\overline{v}^2$ & $e$ & $s_2$ & $v^2$ & $v^2$ & $v^2$ & $v^2$ & $v^2$ & $v^2$ & $v^2$ & $v^2$ \\
\hline
\end{tabular}
\caption{Examples of $(w_k,\overline{v}^l_k)$.}
\end{table}

Let $l=5=\ell(v)$ and $j_5=1$. Then
\[
I_1=\{(1,1),\fbox{(1,2)},\fbox{(1,3)},(1,4)\}=\{1,\fbox{5},\fbox{8},10\},
\]
and
\[
I_{1,0}^5=\{1\},\qquad
I_{1,1}^5=\{5\},\qquad
I_{1,2}^5=\{8,10\}.
\]
It follows that
\[
I^5_{1,2}[-2]=\{(1,1),(1,2)\}=\{1,5\}.
\]
The unselected color-$1$ occurrence numbers are $1$ and $4$. Hence
\[
I_{1,\leq5}=\{(1,1),(1,2)\}=\{1,5\},\qquad
\Phi_1^5(1,1)=(1,1),\quad\Phi_1^5(1,2)=(1,4).
\]
In the original position labels, the map is $1\mapsto1$ and $5\mapsto10$.
\end{example}

\section{Kac--Moody open Richardson varieties}\label{sec:kacmoody}

Let $I$ be a finite set and let $C=(c_{ij})_{i,j\in I}$ be a symmetric generalized Cartan matrix. Fix a simply connected integral Kac--Moody realization
\[
\mathcal{D} := (I, C, X, Y, \{\alpha_i\}_{i\in I}, \{\alpha_i^\vee\}_{i\in I}),
\]
where $X$ and $Y$ are free abelian groups of finite rank in perfect duality, the simple roots $\alpha_i\in X$ and simple coroots $\alpha_i^\vee\in Y$ are each linearly independent, and
\[
\langle \alpha_i, \alpha_j^\vee \rangle = c_{ij}.
\]

Choose fundamental weights $\varpi_i\in X$ satisfying
\[
\langle\varpi_i,\alpha_j^\vee\rangle=\delta_{ij}.
\]
When $C$ is singular, the realization includes the additional lattice directions needed to keep the simple roots and coroots independent; we do not identify $X$ or $Y$ with the span of the displayed weights or coroots. Set
\[
\begin{gathered}
Q:=\bigoplus_{i\in I}\ZZ\alpha_i,\qquad
Q^+:=\bigoplus_{i\in I}\ZZ_{\geq0}\alpha_i,\\
X^+:=\{\lambda\in X\mid\langle\lambda,\alpha_i^\vee\rangle\geq0\text{ for all }i\in I\}.
\end{gathered}
\]
We use the $W$-invariant symmetric form on the realization normalized by
$(\alpha_i,\alpha_j)=c_{ij}$ and
$(\lambda,\alpha_i)=\langle\lambda,\alpha_i^\vee\rangle$.

The Weyl group $W$ is generated by the simple reflections $s_i$ for $i \in I$, acting on $X$ via the bilinear pairing $\langle \cdot, \cdot \rangle$. We denote by $\ell(\cdot)$ the length function on $W$. The set of real roots is defined by
\[
\Delta_{\mathrm{re}} := \{ w(\alpha_i) \mid i \in I,\; w \in W \},
\]
with the disjoint decomposition
\[
\Delta_{\mathrm{re}} = \Delta_{\mathrm{re}}^+ \sqcup \Delta_{\mathrm{re}}^- ,
\qquad
\Delta_{\mathrm{re}}^+:=\Delta_{\mathrm{re}}\cap Q^+,\quad
\Delta_{\mathrm{re}}^-:=-\Delta_{\mathrm{re}}^+.
\]
For the full root system $\Delta$ (including imaginary roots), put
$\Delta^+:=\Delta\cap(Q^+\setminus\{0\})$ and
$\Delta^-:=-\Delta^+$. The inversion set $\Delta^+\cap w\Delta^-$
is finite and consists of real roots.

The \emph{minimal Kac--Moody group} $G$ associated with the root datum $\mathcal{D}$ over $\CC$ is the group generated by the torus
\[
T := Y \otimes_\ZZ \CC^*
\]
and the root subgroups $U_\alpha \cong \CC$ for each real root $\alpha \in \Delta_{\mathrm{re}}$, subject to the Tits relations. Let $U^+$ (resp.\ $U^-$) denote the subgroup of $G$ generated by the root subgroups $U_\alpha$ (resp.\ $U_{-\alpha}$) for $\alpha \in \Delta_{\mathrm{re}}^+$. Define the Borel subgroups
\[
B^+ := \langle T, U^+ \rangle, \qquad
B^- := \langle T, U^- \rangle .
\]

Fix a pinning with simple root homomorphisms $x_i,y_i$ and standard
representatives $\dot s_i=x_i(1)y_i(-1)x_i(1)$. These satisfy the braid
relations, so $\dot w:=\dot s_{i_1}\cdots\dot s_{i_\ell}$ is independent
of the reduced expression $w=s_{i_1}\cdots s_{i_\ell}$.

Let $\mathcal{B} := G/B^+$ be the flag variety. For $w, v \in W$, define the \emph{Schubert cell}
\[
\mathring{\mathcal{B}}_w := B^+ \dot{w} B^+/B^+
\]
and the \emph{opposite Schubert cell}
\[
\mathring{\mathcal{B}}^{\,v} := B^- \dot{v} B^+/B^+ .
\]
By the Bruhat decomposition and the Birkhoff decomposition, one has
\[
\mathcal{B}
= \bigsqcup_{w \in W} \mathring{\mathcal{B}}_w
= \bigsqcup_{v \in W} \mathring{\mathcal{B}}^{\,v}.
\]

Define the open Richardson cell by
\[
\mathring{\mathcal{B}}_{v,w} := \mathring{\mathcal{B}}_w \cap \mathring{\mathcal{B}}^{\,v}.
\]
It is well known that $\mathring{\mathcal{B}}_{v,w} \neq \varnothing$ if and only if $v \le w$ in the Bruhat order.

For $v\leq w$, choose a reduced expression $\overline w$ and its leftmost subexpression $\overline v$ as in Section~\ref{sec:leftmost}. Recall the seed $\mathbf s(\overline v,\overline w)$ from Definition~\ref{def:subwordmutation}.

\begin{theorem}[{\cite[Theorem~5.13]{bao2025upper}}]\label{thm:classical-upper-richardson}
    The coordinate ring of the open Richardson variety has the upper cluster algebra realization
    \[
    \CC[\mathring{\mathcal B}_{v,w}]\cong U(\overline v,\overline w).
    \]
    Here frozen variables are invertible in the upper cluster algebra, as in Section~\ref{sec:cluster}.
\end{theorem}

\section{Determinantial modules}\label{sec:quiverhecke}
\subsection{Quantum coordinate rings}
For the symmetric Cartan datum fixed above, let \(U_q(\mathfrak g)\)
be the quantum group generated by \(e_i,f_i,q^h\), with \(i\in I\)
and \(h\in Y\). Its subalgebra \(U_q(\mathfrak n)=U_q^+(\mathfrak g)\)
is generated by the \(e_i\).
Let \(\phi\) be the antiautomorphism interchanging \(e_i,f_i\)
and fixing \(q^h\).

For \(\lambda\in X^+\), let \(V(\lambda)\) be the integrable
irreducible highest weight module with highest weight vector
\(u_\lambda\). Use its restricted dual
\[
V(\lambda)^\vee
:=\bigoplus_{\nu\in X}\operatorname{Hom}_{\mathbb Q(q)}
       (V(\lambda)_\nu,\mathbb Q(q)).
\]
For \(f\in V(\lambda)^\vee\) and \(u\in V(\lambda)\), put
\(c_{f,u}(x)=f(xu)\).
The quantum coordinate ring \(A_q(\mathfrak g)\) is the span of
these integrable matrix coefficients inside the full linear dual
of \(U_q(\mathfrak g)\), with multiplication induced by the standard
coproduct. The Peter--Weyl decomposition gives
\[
\bigoplus_{\lambda\in X^+}V(\lambda)^\vee\otimes V(\lambda)
\xrightarrow{\sim} A_q(\mathfrak g),\qquad
f\otimes u\longmapsto c_{f,u};
\]
see \cite{kashiwara1993global} and
\cite[Section~8]{kang2018monoidal} for the bimodule conventions.
The restricted dual is essential here: in Kac--Moody type one
cannot replace these coefficients by those of finite-dimensional
representations.

Equip \(V(\lambda)\) with the symmetric contravariant form
\[
(u_\lambda,u_\lambda)_\lambda=1,\qquad
(xu,z)_\lambda=(u,\phi(x)z)_\lambda.
\]
For \(\mu\in W\lambda\), let \(u_\mu\) denote its normalized
extremal vector.

\subsubsection{Quantum minors}
For \(w,v\in W\) and \(\lambda\in X^+\), define
\[
\Delta\binom{w\lambda}{v\lambda}(x)
:=(u_{v\lambda},xu_{w\lambda})_\lambda
\qquad (x\in U_q(\mathfrak g)).
\]
This is the extremal-vector convention of
\cite[Section~9.1]{kang2018monoidal}; in particular, the minor
depends only on its two weights.

Let \(A_q(\mathfrak n)\) be the root-graded dual
\[
\bigoplus_{\gamma\in Q^+}
\operatorname{Hom}_{\mathbb Q(q)}
(U_q(\mathfrak n)_\gamma,\mathbb Q(q)),
\]
with multiplication dual to the usual twisted coproduct on
\(U_q(\mathfrak n)\). The restriction map
\[
p_{\mathfrak n}:A_q(\mathfrak g)\longrightarrow A_q(\mathfrak n),
\qquad p_{\mathfrak n}(f)(x)=f(x),
\]
defines the unipotent quantum minors
\[
D\binom{w\lambda}{v\lambda}
:=p_{\mathfrak n}\!\left(\Delta\binom{w\lambda}{v\lambda}\right).
\]

Set \(J_\lambda=\{j\in I:\langle\lambda,\alpha_j^\vee\rangle=0\}\).
We give \(W/W_{J_\lambda}\) its Bruhat order and write
\(w\lambda\preceq v\lambda\) when
\(vW_{J_\lambda}\leq wW_{J_\lambda}\). The precise criterion is
\cite[Lemma~9.1.4]{kang2018monoidal}
\begin{equation}\label{eq:Dneq0}
D\binom{w\lambda}{v\lambda}\ne0
\quad\Longleftrightarrow\quad
vW_{J_\lambda}\leq wW_{J_\lambda}
\quad\Longleftrightarrow\quad
w\lambda\preceq v\lambda.
\end{equation}
Thus \(v\leq w\) implies non-vanishing, but the converse for
arbitrary Weyl group representatives requires regular \(\lambda\).
We use \(D\binom{\mu}{\mu}=1\).

\begin{lemma}[{\cite[Lemmas~9.1.1 and~9.1.4]{kang2018monoidal}}]
\label{lem:minorcanonical}
Every nonzero \(D\binom{w\lambda}{v\lambda}\) belongs to the upper
global (dual canonical) basis of \(A_q(\mathfrak n)\).
\end{lemma}

The following proposition plays an important role in this paper.

\begin{proposition}[{\cite[Proposition~9.2.1]{kang2018monoidal}}]\label{thm:Tsystem}
Let $w,v\in W$, assume $vs_i>v$ and $ws_i>w$, and set
$\lambda_i:=\varpi_i+s_i\varpi_i\in X^+$. Then
\[
q^A
D\binom{ws_i\varpi_i}{vs_i\varpi_i}
D\binom{w\varpi_i}{v\varpi_i}
=
q^{-1+B}
D\binom{ws_i\varpi_i}{v\varpi_i}
D\binom{w\varpi_i}{vs_i\varpi_i}
+
D\binom{w\lambda_i}{v\lambda_i},
\]
where
\[
A=\left(vs_i\varpi_i,\,v\varpi_i-w\varpi_i\right),
\qquad
B=\left(v\varpi_i,\,vs_i\varpi_i-w\varpi_i\right).
\]
\end{proposition}

For clarity, the neighboring-color term can be expanded with its
precise power of $q$. Fix an ordering $j_1,\ldots,j_N$ of the multiset
in which $j\ne i$ occurs $-c_{ji}$ times, and put
\[
\kappa_i(w,v)
:=\sum_{a<b}(v\varpi_{j_a},\,v\varpi_{j_b}-w\varpi_{j_b}).
\]
Since $\lambda_i-\sum_{a=1}^N\varpi_{j_a}$ is $W$-invariant,
\cite[Corollary~9.1.3]{kang2018monoidal} gives
\begin{equation}\label{eq:neighbor-minor-normalization}
D\binom{w\lambda_i}{v\lambda_i}
=q^{\kappa_i(w,v)}
 D\binom{w\varpi_{j_1}}{v\varpi_{j_1}}
 \cdots
 D\binom{w\varpi_{j_N}}{v\varpi_{j_N}}.
\end{equation}
The empty product is $1$. All products below use this fixed order;
the unexpanded $\lambda_i$-minor is independent of that choice.

\medskip

We remark that our notation for quantum minors differs from that in \cite{geiss2013cluster}. 
One advantage of our notation is that the above identity can be interpreted in terms of a $2\times 2$ matrix:
\[
\begin{pmatrix}
ws_i\varpi_i & w\varpi_i\\
vs_i\varpi_i & v\varpi_i
\end{pmatrix}.
\]
The left-hand side corresponds to the product of two columns,
while the first term on the right-hand side corresponds to the product of the
off-diagonal and diagonal entries. This product can be represented
diagrammatically as follows:
\[
\begin{tikzcd}
ws_i\varpi_i \ar[dr,dash] & w\varpi_i \ar[dl,dash] \\
vs_i\varpi_i & v\varpi_i
\end{tikzcd}
\]
together with the additional column terms
\[
\begin{pmatrix}
w\varpi_j\\
v\varpi_j
\end{pmatrix}
\quad (j\ne i).
\]

Recall the notations in Section \ref{sec:leftmost}.
We now describe a structural property of the quantum minors associated with the
sequence $\overline{v}^l$. More precisely, we study the relation between the
quantum minors
\[
D\binom{w_p\varpi_{i_p}}{\overline{v}_p^{\,l}\varpi_{i_p}}
\quad\text{and}\quad
D\binom{w_q\varpi_{i_q}}{\overline{v}_q^{\,l-1}\varpi_{i_q}}.
\]
\begin{theorem}\label{thm:Dpllaurent}
For $l\in [m]$, the quantum minor
\[
D\binom{w_{p_l}\varpi_{i_{p_l}}}
{\overline{v}^{\,l}_{p_l}\varpi_{i_{p_l}}}
\]
is, up to a power of $q$, a Laurent monomial in
\[
\left\{
D\binom{w_k\varpi_{i_k}}
{\overline{v}^{\,l-1}_k\varpi_{i_k}}
\mid k<p_l\right\} .
\]
\end{theorem}

\begin{proof}
Set \(i=i_{p_l}\), \(a=v^{l-1}\), \(b=v^l=as_i\), and
\(u=w_{p_l-1}\). Then \(a\leq u\), \(u<us_i=w_{p_l}\), and
\(a<b\). Lemma~\ref{lem:leftmost-prefix} gives \(b\nleq u\).

We first prove the needed vanishing in the appropriate weight orbit.
For \(j\in I\), write \(\pi_j:W\to W/W_{I\setminus\{j\}}\).
If \(j\ne i\), then \(s_i\in W_{I\setminus\{j\}}\), and hence
\[
\pi_j(b)=\pi_j(a)\leq\pi_j(u).
\]
Bruhat order is detected by these maximal-parabolic projections
\cite[Theorem~1.3 and Corollary~1.4]{stembridge2005tight}.
Consequently \(b\nleq u\) forces
\(\pi_i(b)\nleq\pi_i(u)\). Since
\(W_{I\setminus\{i\}}\) stabilizes \(\varpi_i\), the criterion
\eqref{eq:Dneq0} now gives
\[
D\binom{u\varpi_i}{b\varpi_i}=0.
\]

Apply Proposition~\ref{thm:Tsystem} to \(u,a,i\). Its first
right-hand term vanishes, so that
\begin{equation}\label{eq:TsystemD}
q^{(b\varpi_i,a\varpi_i-u\varpi_i)}
D\binom{us_i\varpi_i}{b\varpi_i}
D\binom{u\varpi_i}{a\varpi_i}
=
D\binom{u\lambda_i}{a\lambda_i}.
\end{equation}
By \eqref{eq:neighbor-minor-normalization}, the last minor is a
power of \(q\) times an ordered monomial in
\(D\binom{u\varpi_j}{a\varpi_j}\), \(j\ne i\).

For each \(j\), let \(h_j<p_l\) be the last position of color \(j\)
before \(p_l\), if it exists. No letter of color \(j\) occurs
between \(h_j\) and \(p_l-1\), in either the full word or its selected
subword. Therefore
\[
u\varpi_j=w_{h_j}\varpi_j,\qquad
a\varpi_j=\overline v^{\,l-1}_{h_j}\varpi_j.
\]
If \(h_j\) does not exist, both weights equal \(\varpi_j\), and the
corresponding minor is \(1\). All minors in
\eqref{eq:TsystemD} other than its first factor are thus among the
stated earlier minors, or equal to \(1\). They are nonzero because
\(a\leq u\). Solving for the first factor in the skew field of
fractions proves the assertion, with the indicated power of \(q\).
\end{proof}

\begin{proposition}\label{pro:Tsystemmk}
Fix \(l\in[m]\), write \(p_l=(i,d_l)\), and take
\((i,k)\in I^l_{i,\alpha(i,l)}\setminus\{p_l\}\).
Let \(p\) be the absolute word position of \((i,k)\), and put
\(u=w_{p-1}\), \(a=v^{l-1}\), and \(\lambda_i=\varpi_i+s_i\varpi_i\).
Then
\begin{equation}\label{eq:Tsystemid}
\begin{aligned}
q^{A_{l,k}}
D\binom{w_{(i,k)}\varpi_i}{v^l\varpi_i}
D\binom{w_{(i,k-1)}\varpi_i}{v^{l-1}\varpi_i}
&=
q^{-1+B_{l,k}}
D\binom{w_{(i,k)}\varpi_i}{v^{l-1}\varpi_i}
D\binom{w_{(i,k-1)}\varpi_i}{v^l\varpi_i}\\
&\quad+
D\binom{w_{p-1}\lambda_i}{v^{l-1}\lambda_i},
\end{aligned}
\end{equation}
where
\[
A_{l,k}=(v^l\varpi_i,v^{l-1}\varpi_i-u\varpi_i),
\qquad
B_{l,k}=(v^{l-1}\varpi_i,v^l\varpi_i-u\varpi_i).
\]
The last term equals
\[
q^{\kappa_i(w_{p-1},v^{l-1})}
\prod_{a=1}^N
D\binom{w_{p-1}\varpi_{j_a}}{v^{l-1}\varpi_{j_a}}
\]
in the ordering of \eqref{eq:neighbor-minor-normalization}.
\end{proposition}

\begin{proof}
The word contains no letter of color \(i\) between the preceding
occurrence \((i,k-1)\) and position \(p\). Thus
\[
w_{p-1}\varpi_i=w_{(i,k-1)}\varpi_i,
\qquad w_p=w_{p-1}s_i.
\]
Moreover, both occurrences are at or after \(p_l\), so their lower
endpoints at stages \(l-1,l\) are respectively \(v^{l-1}\) and \(v^l\).
Proposition~\ref{thm:Tsystem}, applied to \(u=w_{p-1}\) and \(a=v^{l-1}\),
gives \eqref{eq:Tsystemid}. The neighboring-color expansion follows
from \eqref{eq:neighbor-minor-normalization}. Notice that its prefix
is \(w_{p-1}\); it cannot generally be replaced by \(w_{(i,k-1)}\).
\end{proof}

\begin{example}[The adjacent-occurrence identity in type \(A_2\)]
\label{ex:Tsystem-A2}
Take \(\overline w=(1,2,1)\), \(v=s_1\), and \(l=1\). The next
occurrence of color \(1\) has absolute position \(p=3\), so the
neighboring-color term uses \(w_2=s_1s_2\). At \(q=1\) in the
deformation parameter, write
\[
n=\begin{pmatrix}1&x&z\\0&1&y\\0&0&1\end{pmatrix}.
\]
With the matrix-coefficient convention above, the four
\(\varpi_1\)-minors in \eqref{eq:Tsystemid} specialize to
\(y,x,z,1\), respectively, while
\[
D\binom{s_1s_2\varpi_2}{\varpi_2}(n)
=\det\begin{pmatrix}x&z\\1&y\end{pmatrix}=xy-z.
\]
Thus the identity is \(yx=z+(xy-z)\).
This example also distinguishes the absolute predecessor \(p-1=2\)
from the previous occurrence of color \(1\), which is position \(1\).
\end{example}

We may interpret Proposition~\ref{pro:Tsystemmk} pictorially as follows:
\begin{center}
\begingroup\scriptsize
\begin{equation}\label{dia:vl}
\begin{tikzcd}[column sep=2.2em,row sep=1.8em]
v^{l-1} 
& (i,d) \arrow[dr, draw=red, -] 
& (i,d+1) \arrow[dr, draw=red, -] 
& \cdots 
& (i,d+k-1) \arrow[dr, draw=red, -] 
& (i,d+k) \arrow[dr, draw=red, -] 
& \cdots \\
v^{l} 
& (i,d) \arrow[ur, draw=blue, -] 
& (i,d+1) \arrow[ur, draw=blue, -] 
& \cdots 
& (i,d+k-1) \arrow[ur, draw=blue, -] 
& (i,d+k) \arrow[ur, draw=blue, -] 
& \cdots
\end{tikzcd}
\end{equation}
\endgroup
\end{center}
The pair of terms connected by the red lines corresponds to the left-hand side of~\eqref{eq:Tsystemid}. 
The pair of terms connected by the blue lines corresponds to the first term on the right-hand side of~\eqref{eq:Tsystemid}.

\subsubsection{Cluster structures on quantum unipotent coordinate rings}
Let \(\overline w=(i_1,\ldots,i_r)\) be a reduced expression of \(w\), and set
\[
E_k^{\overline w}:=T_{i_1}\cdots T_{i_{k-1}}(e_{i_k})
\qquad(k\in[r]),
\]
where the \(T_i\) are Lusztig's braid group automorphisms
\cite{lusztig2010introduction}. The quantum nilpotent algebra
\(U_q(\mathfrak n(w))\) is generated by these root vectors.
We use \(A_q(\mathfrak n(w))\) for its quantum unipotent coordinate
realization inside \(A_q(\mathfrak n)\), with the standard multiplication
and upper global basis; it is not the unrestricted linear dual of
\(U_q(\mathfrak n(w))\).

Put \(\mathbb A=\mathbb Z[q^{\pm1}]\) and retain
\(\KK=\mathbb Z[q^{\pm1/2}]\) from Section~\ref{sec:cluster}.
The \(\mathbb A\)-form \(A_q(\mathfrak n)_{\mathbb A}\) is the dual
of the divided-power integral form, and
\[
A_q(\mathfrak n(w))_{\mathbb A}
:=A_q(\mathfrak n(w))\cap A_q(\mathfrak n)_{\mathbb A},
\qquad
A_q(\mathfrak n(w))_{\KK}
:=\KK\otimes_{\mathbb A}A_q(\mathfrak n(w))_{\mathbb A}.
\]
Set
\[
D_k:=D\binom{w_k\varpi_{i_k}}{\varpi_{i_k}},
\qquad
\beta_k:=\varpi_{i_k}-w_k\varpi_{i_k},
\qquad
\widetilde D_k:=q^{-(\beta_k,\beta_k)/4}D_k.
\]
The quantum seed associated with \(\overline w\) is
\begin{equation}\label{eq:normalized-word-seed}
\mathbf s_q(\overline w)
:=
\bigl((\widetilde D_k)_{k\in[r]},\Lambda_{\overline w},
       B_{\overline w},J_{\mathrm{uf}}\bigr).
\end{equation}
Here $J_{\mathrm{uf}}$ consists of positions which are not the final
occurrence of their color, as in Section~\ref{sec:wordseed}.
The normalization is that of
\cite[Theorem~11.1.4]{kang2018monoidal}. We distinguish it from the
classical word seed only when the coefficient parameter matters.

\begin{theorem}[{\cite{geiss2013cluster,kang2018monoidal}}]
The assignment of initial variables in
\eqref{eq:normalized-word-seed} extends to an isomorphism
\[
\overline{\mathcal A}_q(\mathbf s_q(\overline w))
\xrightarrow{\sim}A_q(\mathfrak n(w))_{\KK}.
\]
Here the bar denotes the quantum cluster algebra with non-inverted
frozen variables. Inverting the frozen minors gives the corresponding
localized isomorphism. Both assertions are over \(\KK\); extending
scalars to \(\mathbb Q(q^{1/2})\) gives their field-coefficient versions.
\end{theorem}

\subsection{Quiver Hecke algebras}

In this section, we briefly recall the definition of quiver Hecke algebras and the notion of determinantial modules.

\subsubsection{Definition of quiver Hecke algebras}

Let $Q=(I,Q_1)$ be a quiver whose associated Cartan matrix $C$ is symmetric. 
For $i,j\in I$, let $m_{ij}$ denote the number of arrows from $i$ to $j$. 
We define a polynomial $q_{i,j}(u,v)\in \CC[u,v]$ by
\[
q_{i,j}(u,v)=
\begin{cases}
0, & \text{if } i=j,\\
(v-u)^{m_{ij}}(u-v)^{m_{ji}}, & \text{if } i\neq j.
\end{cases}
\]

Let
\[
\alpha=\sum_{i\in I}n_i\,\alpha_i\in Q^+
\qquad\text{with}\qquad
n=|\alpha|=\sum_{i\in I}n_i.
\] We write $\langle I\rangle_\alpha$ for the set of words $\mathbf{i}=(i_1,\dots,i_n)$ of weight $\alpha$, where $n=|\alpha|$.

\begin{definition}

The \emph{quiver Hecke algebra} $R(\alpha)$ is the associative $\CC$-algebra generated by
\[
\{1_{\mathbf{i}}\}_{\mathbf{i}\in \langle I\rangle_\alpha}
\cup \{x_1,\dots,x_n\}
\cup \{\tau_1,\dots,\tau_{n-1}\},
\]
with the standard quiver Hecke relations associated with the
polynomials $q_{i,j}$; see \cite[Section~2.1]{kang2018monoidal}.
In particular, the $1_{\mathbf i}$ are mutually orthogonal idempotents
whose sum is $1_\alpha$, the $x_k$ commute, and
$\tau_k^2 1_{\mathbf i}=q_{i_k,i_{k+1}}(x_k,x_{k+1})1_{\mathbf i}$.
\end{definition}

The algebra $R(\alpha)$ carries a natural $\ZZ$-grading given by
\[
\deg(1_{\mathbf{i}})=0,
\qquad
\deg(x_j)=2,
\qquad
\deg(\tau_k1_{\mathbf{i}})
=-(\alpha_{i_k},\alpha_{i_{k+1}}),
\]
where $(\cdot,\cdot)$ is the symmetric bilinear form associated with the Cartan matrix $C$.

For $\alpha,\beta\in Q^+$, let
$e(\alpha,\beta):=\sum_{\mathbf i\in\langle I\rangle_\alpha,
\,\mathbf j\in\langle I\rangle_\beta}1_{\mathbf i\mathbf j}$
be the concatenation idempotent in $R(\alpha+\beta)$.

Let $R(\alpha)\text{-}\gmod$ denote the category of finite-dimensional graded $R(\alpha)$-modules. For
\[
M\in R(\alpha)\text{-}\gmod,
\qquad
N\in R(\beta)\text{-}\gmod,
\]
their convolution product is defined by
\[
M\circ N
:=
R(\alpha+\beta)e(\alpha,\beta)\otimes_{R(\alpha)\otimes R(\beta)}(M\boxtimes N).
\]

For $M\in R(\beta)\text{-}\gmod$, its graded dual is
\[
M^*:=\Hom_\CC(M,\CC),\qquad
(r\cdot f)(u):=f(\psi(r)u)
\qquad
(r\in R(\beta),\ u\in M),
\]
where $\psi$ is the algebra anti-involution fixing all the generators.

A simple $R(\beta)$-module $M$ is called \emph{self-dual} if $M^*\cong M$.

Now set
\[
R\text{-}\gmod
:=
\bigoplus_{\alpha\in Q^+}R(\alpha)\text{-}\gmod,
\]
and denote its graded Grothendieck ring by $K_0(R\text{-}\gmod)$,
an $\mathbb A$-algebra with $[qM]=q[M]$.
For a graded monoidal subcategory $\mathscr C$, we write
\[
K(\mathscr C):=\KK\otimes_{\mathbb A}K_0(\mathscr C),
\qquad
K_{\mathrm{cl}}(\mathscr C)
:=\mathbb C\otimes_{\KK,\,q^{1/2}\mapsto1}K(\mathscr C).
\]
The latter is the complexified ungraded Grothendieck ring.
For a module $M\in R(\alpha)\text{-}\gmod$, we use the root degree
\[
\wt(M)=\alpha.
\]

\begin{theorem}[{\cite[Theorems~2.1.2 and~2.1.4]{kang2018monoidal}}]\label{thm:isoK0Phi}
There is an $\mathbb A$-algebra isomorphism
\[
\Phi:\,K_0(R\text{-}\gmod)\xrightarrow{\sim} A_q(\mathfrak n)_{\mathbb A}
\]
under which the classes of self-dual simple modules correspond to the
upper global basis. We use the same symbol for its scalar extension
to $\KK$. The categorification isomorphism and the identification of
simple modules with the global basis are separate assertions; the
latter uses symmetry and characteristic zero.
\end{theorem}

In particular, by Lemma~\ref{lem:minorcanonical}, for each nonzero quantum minor
\[
D\binom{w\lambda}{v\lambda},
\]
there exists a corresponding self-dual simple module, which we denote by
\[
M\binom{w\lambda}{v\lambda}.
\]
We call these modules \emph{determinantial modules}. They are real
\cite[Lemma~10.2.2]{kang2018monoidal}; equal endpoints give the unit
module. A module denotes its class only inside a Grothendieck-ring
formula, where square brackets will be used. Isomorphisms involving
unnormalized convolution products are understood up to integral
grading shift unless a shift is displayed. The normalized quantum
variable attached to a self-dual simple $M$ is
$q^{-(\mathrm{wt}(M),\mathrm{wt}(M))/4}\Phi([M])$.

\subsubsection{Determinantial modules and commutation}

We say that two simple modules $M$ and $N$ \emph{strongly commute} if 
$M\circ N$ is simple. 
A simple module $M$ is called \emph{real} if $M\circ M$ is simple.

For $M\in R(\alpha)\text{-}\operatorname{gmod}$ and 
$N\in R(\beta)\text{-}\operatorname{gmod}$, 
Kang–Kashiwara–Kim~\cite{kang2018symmetric} constructed an intertwining 
$R(\alpha+\beta)$-module homomorphism
\[
\mathbf{r}_{M,N}: M\circ N \longrightarrow N\circ M,
\]
called the \emph{$R$-matrix}.

We define
\[
\Lambda(M,N):=\deg(\mathbf{r}_{M,N}),
\qquad
\mathfrak{d}(M,N)
:=\frac{1}{2}\big(\Lambda(M,N)+\Lambda(N,M)\big).
\]

\begin{lemma}[Additivity and inheritance of commutation]
\label{lem:d-additivity}
Let $L$ be a real simple module, and let $N_1,\ldots,N_t$ be
pairwise strongly commuting real simple modules. Then
\[
\mathfrak d(L,N_1\circ\cdots\circ N_t)
=\sum_{a=1}^t\mathfrak d(L,N_a).
\]
In particular, $L$ strongly commutes with their product if and only
if it strongly commutes with each factor. More generally, if a
simple module $P$ is a convolution factor of a simple module $N$,
then every simple module strongly commuting with $N$ also strongly
commutes with $P$.
\end{lemma}

\begin{proof}
The additivity of both $R$-matrix degrees follows from
\cite[Lemma~3.1.5(ii)]{kang2018monoidal}. Each summand is nonnegative,
and, when one factor is real, its vanishing is equivalent to strong
commutation \cite[Lemma~3.2.3]{kang2018monoidal}.
For the last assertion write $N\simeq P\circ Q$, up to shift.
If $L\circ P$ had a proper nonzero submodule, its convolution with
the nonzero module $Q$ would remain proper and nonzero, because
convolution is exact and sends nonzero modules to nonzero modules
by the KLR shuffle formula. This would contradict simplicity of
$L\circ N$. Thus $L\circ P$ is simple.
\end{proof}

The following proposition provides a family of strongly commuting pairs of determinantial modules.

\begin{proposition}[{\cite[Proposition~10.2.3]{kang2018monoidal}}]\label{pro:commutsstt}
Let $\lambda,\mu\in X^+$ be dominant weights. 
Let $s,s',t,t'$ be Weyl group elements such that
\[
\ell(s's)=\ell(s')+\ell(s),
\qquad
\ell(t't)=\ell(t')+\ell(t),
\]
\[
s's\lambda \preceq t'\lambda,
\qquad
s'\mu \preceq t't\mu.
\]
Then the simple modules
\[
M\binom{s's\lambda}{t'\lambda}
\quad\text{and}\quad
M\binom{s'\mu}{t't\mu}
\]
strongly commute.
\end{proposition}
Following \cite{kashiwara2018monoidal}, we have the following result.

\begin{proposition}[{\cite[Theorem~4.10]{kashiwara2018monoidal}}]\label{pro:commutewv}
For any $l\in [0,\ell(v)]$, the indexed family of determinantial modules
\[
\mathscr{T}_l:=\left\{
M\binom{w_p\varpi_{i_p}}{\overline{v}^l_p\varpi_{i_p}}
\right\}_{p\in [r]}
\]
mutually strongly commute.
\end{proposition}
\begin{proof}
For $l=0$, the empty subexpression is leftmost for $e$; for $l>0$,
Lemma~\ref{lem:leftmost-prefix} shows that the selected positions
$p_1,\ldots,p_l$ form the leftmost subexpression for $v^l$.
Thus the assertion is precisely the cited theorem applied to
$(w,v^l)$, with the indicated fundamental weights.
\end{proof}

The following lemma shows that, except for the last segment of color $i$
determined by $p_l$, the determinantial modules corresponding to
$\overline{v}^{\,l}$ and $\overline{v}^{\,l-1}$ coincide.
This observation will be used repeatedly in the subsequent arguments.

\begin{lemma}\label{lem:l=l-1}
For $l\in [\ell(v)]$ with $i_{p_l}=i$, and for any $q\notin I^l_{i,\alpha(i,l)}$, we have
\begin{equation}\label{eq:ml=l-1}
    M\binom{w_q\varpi_{i_q}}{\overline{v}^l_{q}\varpi_{i_q}}
    =
    M\binom{w_q\varpi_{i_q}}{\overline{v}^{\,l-1}_{q}\varpi_{i_q}}.
\end{equation}
\end{lemma}

\begin{proof}
Recall from Section~\ref{sec:leftmost} that
\[
\overline{v}^l_q=\overline{v}_q^{\,l-1}
\qquad \text{for all } q<p_l .
\]
Therefore the equality~\eqref{eq:ml=l-1} holds immediately for $q<p_l$.

Now suppose $q\ge p_l$ but $q\notin I^l_{i,\alpha(i,l)}$.
By definition, $I^l_{i,\alpha(i,l)}$ consists of all vertices
$(i,k)$ with $k\ge t_{\alpha(i,l)}$, that is, all indices
$p\ge p_l$ of color $i$.
Hence $q$ must have color different from $i$, say $q=(j,k)$ with $j\neq i$.

Since $v^l=v^{l-1}s_i$, we have
\[
\overline{v}_q^{\,l}\varpi_j
=
\overline{v}_q^{\,l-1}s_i\varpi_j.
\]
Because $j\neq i$, the simple reflection $s_i$ fixes the fundamental weight
$\varpi_j$, and therefore
\[
\overline{v}_q^{\,l}\varpi_j
=
\overline{v}_q^{\,l-1}\varpi_j .
\]
This implies the equality of the corresponding determinantial modules,
and hence~\eqref{eq:ml=l-1} holds for such $q$ as well.
\end{proof}

Recall that for $l\in [m]$ we set $(i,d_l)=p_l$ and consider the set
$I_{i,\alpha(i,l)}^l\setminus\{p_l\}$.
For $0\leq k\leq n_i-d_l$, we define
\begin{equation}
\begin{split}
Y_{l,k}
&=
\Set{(i,s)\in I^l_{i,\alpha(i,l)}}{d_l+1\le s\le d_l+k},\\
T_{l,k}
&=
[r]\setminus
\left\{(i,s)\in I_i\mid d_l\le s\le d_l+k-1\right\}.
\end{split}
\end{equation}
In particular, $Y_{l,0}=\varnothing$ and $T_{l,0}=[r]$.

Thus $Y_{l,k}$ consists of the vertices of color $i$ in $I_{i,\alpha(i,l)}^l$
whose indices lie between $d_l+1$ and $d_l+k$. In other words, it records the
initial segment of $I_{i,\alpha(i,l)}^l$ following the vertex $p_l=(i,d_l)$,
up to the vertex $(i,d_l+k)$.

On the other hand, $T_{l,k}$ is obtained from the full index set $[r]$ by
removing the vertices $(i,s)$ with $d_l\le s\le d_l+k-1$.
Equivalently, $T_{l,k}$ may be viewed as the complement in $[r]$ of the shifted
set $Y_{l,k}[-1]$. Set 
\[I_{i,\alpha(i,l),\leq k}^l:=\{(i,s)\in I_i \mid b_l + 1 \leq s \leq b_l + k\}\]
It follows from $d_l = b_l + \alpha(i,l)$ that
\begin{equation}
(\Phi_i^l)^{-1}(Y_{l,k})
=
I_{i,\alpha(i,l),\leq k}^l,
\end{equation}
and
\begin{equation}
\Phi_i^{l-1}\bigl(I_{i,\alpha(i,l),\leq k}^l\bigr)
=
Y_{l,k}[-1].
\end{equation}
Therefore, there exists a well-defined map
\begin{equation}\label{eq:Phijlk}
\Phi_j^{l,k} : I_{j,\leq l-1} \to Y_{l,k} \sqcup T_{l,k}
\end{equation}
such that
\[
\Phi_j^{l,k}(j,s)
=
\begin{cases}
\Phi_j^l(j,s), & \text{if } (j,s)\in I_{i,\alpha(i,l),\leq k}^l,\\
\Phi_j^{l-1}(j,s), & \text{otherwise}.
\end{cases}
\]

\begin{proposition}\label{cor:commute}
For each $(i,d_l+k)\in I_{i,\alpha(i,l)}^l\setminus\{p_l\}$, the modules
\begin{equation}\label{eq:commutevlvl-1}
M\binom{w_{(i,d_l+k)}\varpi_{i}}{\overline{v}^l_{(i,d_l+k)}\varpi_i}
\quad\text{and}\quad
M\binom{w_p\varpi_{i_p}}{\overline{v}^{\,l-1}_p\varpi_{i_p}} \text{ strongly commute for all $p\in T_{l,k}$.}
\end{equation}
\end{proposition}

\begin{proof}
First consider the case $p\notin I^l_{i,\alpha(i,l)}$.
By Lemma~\ref{lem:l=l-1}, we have
\[
M\binom{w_p\varpi_{i_p}}{\overline{v}^{\,l-1}_p\varpi_{i_p}}
=
M\binom{w_p\varpi_{i_p}}{\overline{v}^{\,l}_p\varpi_{i_p}}.
\]
Therefore equation~\eqref{eq:commutevlvl-1} follows immediately from
Proposition~\ref{pro:commutewv}.

Next assume that $p\in T_{l,k}\cap I_{i,\alpha(i,l)}^l$.
By definition, this means
\[
p=(i,s)\quad\text{with}\quad d_l+k\le s\le n_i .
\]
For such $p$, we have
\[
\overline{v}_p^{\,l-1}s_i
=
v^{l-1}s_i
=
v^l
=
\overline{v}_p^{\,l}.
\]
Moreover, since $p=(i,s)$ with $s\ge d_l+k$, we can write
\[
w_p = w_{(i,d_l+k)}\,u
\]
with $\ell(w_p)=\ell(w_{(i,d_l+k)})+\ell(u)$.
Applying Proposition~\ref{pro:commutsstt} with
$s'=w_{(i,d_l+k)}$, $s=u$, $t'=v^{l-1}$, and $t=s_i$ yields
equation~\eqref{eq:commutevlvl-1}.

This completes the proof.
\end{proof}

\begin{definition}
Let $v\le w$ be Weyl group elements. 
For $l\in [\ell(v)]$, write $p_l=(i,d_l)$. 
For each $0\leq k\leq n_i-d_l$, define
\[
\mathscr{T}_{l,k}
:=
\left\{
M\binom{w_p\varpi_{i_p}}{\overline{v}^{\,l-1}_p\varpi_{i_p}}
\right\}_{p\in T_{l,k}}
\ \cup\
\left\{
M\binom{w_q\varpi_{i}}{\overline{v}^{\,l}_q\varpi_i}
\right\}_{q\in Y_{l,k}}.
\]
\end{definition}
One of the motivations for this definition is to relate it to the mutation
sequence $\widetilde{\mu}_l$ introduced in Definition~\ref{def:subwordmutation}.
By Lemma~\ref{lem:l=l-1}, the determinantial modules in $\mathscr{T}_{l,k}$
that differ from those in $\mathscr{T}_l$ are precisely the elements in
\[
\left\{
M\binom{w_p\varpi_{i_p}}{\overline{v}^{\,l-1}_p\varpi_{i_p}}
\;\middle|\;
p=(i,s)\ \text{with}\ d_l+k\le s\le n_i
\right\}.
\]


Note that
\begin{equation}\label{eq:Tlk-1Ylk}
T_{l,k-1}=T_{l,k}\cup\{(i,d_l+k-1)\},
\qquad
Y_{l,k}=Y_{l,k-1}\cup\{(i,d_l+k)\}.
\end{equation}
Thus the sets $T_{l,k}$ and $Y_{l,k}$ are obtained from $T_{l,k-1}$ and
$Y_{l,k-1}$ by removing and adding a single vertex, respectively.

Consequently, we obtain the recursive relation
\begin{equation}\label{eq:Tlkrecursion}
\mathscr{T}_{l,k}
\setminus
\left\{
M\binom{w_{(i,d_l+k)}\varpi_{i}}{\overline{v}^{\,l}_{(i,d_l+k)}\varpi_i}
\right\}
=
\mathscr{T}_{l,k-1}
\setminus
\left\{
M\binom{w_{(i,d_l+k-1)}\varpi_i}{\overline{v}^{\,l-1}_{(i,d_l+k-1)}\varpi_i}
\right\},
\end{equation}
for $1\leq k\leq n_i-d_l$, where by convention we set
$\mathscr{T}_{l,0}=\mathscr{T}_{l-1}$. It is easy to see
\begin{equation}\label{eq:Phil-1Phil}
   \Phi_i^{l-1} (\Phi_i^{l})^{-1}((i,d_l+k))=(i,d_l+k-1).
\end{equation}

Combining Lemma~\ref{lem:l=l-1} with the above discussion, we further obtain
\begin{equation}\label{eq:Tlni-d}
\mathscr{T}_{l,n_i-d_l}\setminus
\left\{
M\binom{w_{(i,n_i)}\varpi_i}{\overline{v}^{\,l-1}_{(i,n_i)}\varpi_i}
\right\}
=
\mathscr{T}_{l}\setminus
\left\{
M\binom{w_{p_l}\varpi_i}{\overline{v}^{\,l}_{p_l}\varpi_i}
\right\}.
\end{equation}

\begin{proposition}\label{pro:Tlkcommute}
Fix $l\in[m]$ and put $i=i_{p_l}$.
For each $0\leq k\leq n_i-d_l$, the family $\mathscr{T}_{l,k}$ mutually strongly commutes. 
Moreover, for $1\leq k\leq n_i-d_l$,
\begin{equation}\label{eq:dMlml-1}
\mathfrak{d}\!\left(
M\binom{w_{(i,d_l+k)}\varpi_{i}}{\overline{v}^{\,l}_{(i,d_l+k)}\varpi_i},
M\binom{w_{(i,d_l+k-1)}\varpi_{i}}{\overline{v}^{\,l-1}_{(i,d_l+k-1)}\varpi_i}
\right)
=1.
\end{equation}
\end{proposition}
\begin{proof}
For $k=0$, mutual strong commutativity is
Proposition~\ref{pro:commutewv}. For $k>0$, it follows from Propositions
\ref{cor:commute} and~\ref{pro:commutewv}, since
\(T_{l,k}\subset T_{l,k'}\) for \(k'<k\).

For the second assertion let \(h,p\) be the absolute positions of
\((i,d_l+k-1)\), \((i,d_l+k)\), respectively, and put
\(a=v^{l-1}\), \(b=as_i=v^l\), and \(u=w_{p-1}\).
The subword property gives
\[
a\leq b\leq w_h\leq u,\qquad u<us_i=w_p,
\qquad u\varpi_i=w_h\varpi_i.
\]
Set
\[
P=M\binom{w_p\varpi_i}{b\varpi_i},
\qquad Q=M\binom{w_h\varpi_i}{a\varpi_i}.
\]
Apply \cite[Proposition~10.3.2(ii)]{kang2018monoidal}
with \(x=u\), \(z=a\), and \(y=s_i\). Its length and Bruhat
hypotheses hold by the displayed relations, so
\(\mathfrak d(P,Q)\leq1\).

Both terms on the right of \eqref{eq:Tsystemid} are nonzero
classes of modules, up to powers of \(q\). Indeed, all their
determinantial endpoints are comparable because \(b\leq u\).
The first term is a convolution product of two nonzero
determinantial modules; the second is the nonzero determinantial
module \(M\binom{u\lambda_i}{a\lambda_i}\).
The isomorphism \(\Phi\) therefore turns \eqref{eq:Tsystemid} into
\[
[P][Q]=q^c[U]+q^d[V]
\]
for some nonzero modules \(U,V\) and integers \(c,d\).
Positivity in the simple-class basis implies that \(P\circ Q\)
is not simple. Since \(P,Q\) are real, this gives
\(\mathfrak d(P,Q)>0\), and hence \(\mathfrak d(P,Q)=1\).
\end{proof}

\begin{proposition}\label{lem:comwvs}
Fix $l\in[m]$ and put $i=i_{p_l}$. For each $k\in [0,n_i-d_l]$, the family
\[
\mathscr{T}_{l,k}\cup
\left\{
M\binom{w\varpi_j}{v^s\varpi_j}
\;\middle|\;
0\leq s<l-1,\; j\in I
\right\}
\]
mutually strongly commutes.
\end{proposition}

\begin{proof}
Let $0\leq s<l-1$ and $j\in I$, and put $p_0=0$.
We first show that the determinantial modules
$M\binom{w\varpi_j}{v^s\varpi_j}$ strongly commute with all modules in
$\mathscr{T}_{l,k}$.

If $q\le p_s$, then
\[
\overline{v}^{\,l-1}_q=\overline{v}^{\,l}_q=\overline{v}^{\,s}_q.
\]
Hence, by Proposition~\ref{pro:commutewv}, the module
\[
M\binom{w_q\varpi_{i_q}}{\overline{v}^{\,l-1}_q\varpi_{i_q}}
\]
strongly commutes with $M\binom{w\varpi_j}{v^s\varpi_j}$.

Next consider the case $q>p_s$. In this situation we can write
\[
\overline{v}_q^{\,l-1}=v^s u_1
\quad\text{and}\quad
\overline{v}_q^{\,l}=v^s u_2
\]
for some $u_1,u_2\in W$, with both factorizations length-additive.
Since $w_q$ is a prefix of the fixed reduced expression of $w$, we have
$w_qw'=w$ with $\ell(w)=\ell(w_q)+\ell(w')$.
Therefore Proposition~\ref{pro:commutsstt}
implies that $M\binom{w\varpi_j}{v^s\varpi_j}$ strongly commutes with any
determinantial module in $\mathscr{T}_{l,k}$.

Finally, for any $j,h\in I$ and $0\leq a\leq b<l-1$, the
factorization $v^b=v^a u$ is length-additive.
Apply Proposition~\ref{pro:commutsstt} with
$s'=w$, $s=e$, $t'=v^a$, and $t=u$. It follows that
\[
M\binom{w\varpi_j}{v^a\varpi_j}
\quad\text{and}\quad
M\binom{w\varpi_h}{v^b\varpi_h}
\]
strongly commute. Hence all modules in the stated family
mutually strongly commute.
\end{proof}

\begin{definition}
Let $\mathscr T=(T_i)_{i\in K}$ and $\mathscr S=(S_j)_{j\in J}$
be strongly commuting families. We say that $\mathscr S$ is
\emph{Laurent-generated} by $\mathscr T$ if each $\Phi([S_j])$
is a power of $q$ times an ordered Laurent monomial in the
$\Phi([T_i])$, in the common skew field of fractions.
We abbreviate this to ``generated'' when no confusion is possible.
This does not by itself assert that the exponents are nonnegative.
\end{definition}

\begin{proposition}\label{pro:Tlni-d}
Fix $l\in [\ell(v)]$ and put $i=i_{p_l}$. The family
\[
\mathscr{T}_{l,n_i-d_l}
\setminus 
\left\{
M\binom{w_{(i,n_i)}\varpi_i}{\overline{v}^{\,l-1}_{(i,n_i)}\varpi_i}
\right\}
\]
generates the commuting family $\mathscr{T}_l$.
\end{proposition}

\begin{proof}
By \eqref{eq:Tlni-d}, it suffices to show that
\[
\Phi\!\left(
\left[M\binom{w_{p_l}\varpi_i}{\overline{v}^{\,l}_{p_l}\varpi_i}\right]
\right)
\]
is a Laurent monomial in the classes of the retained family.

Theorem~\ref{thm:Dpllaurent} expresses it using only determinantial
modules indexed by positions $q<p_l$. These positions are unaffected
in $\mathscr T_{l,n_i-d_l}$ and are distinct from the removed position
$(i,n_i)\geq p_l$. Thus every factor in that expression is retained,
including when $n_i=d_l$ and the mutation sequence is empty.
\end{proof}

\section{Monoidal categorification of quiver Hecke algebras}\label{sec:monoidalcat}

Throughout this section, the Cartan datum is symmetric. We work with
graded Grothendieck rings over $\mathbb Z[q^{\pm1/2}]$, extending scalars
from $\mathbb Z[q^{\pm1}]$ when necessary. Equalities between convolution
products of simple modules are understood up to grading shift unless
a normalization is explicitly specified.
\subsection{Subcategories associated with Weyl group elements}
For $M\in R(\alpha)\text{-}\gmod$, define
\[
\mathbf W(M):=\{\beta\in Q^+\mid
\alpha-\beta\in Q^+,\ e(\beta,\alpha-\beta)M\ne0\},
\]
\[
\mathbf W^*(M):=\{\beta\in Q^+\mid
\alpha-\beta\in Q^+,\ e(\alpha-\beta,\beta)M\ne0\}.
\]

Let $w$ be a Weyl group element, and let $\overline{w}=(i_1\cdots i_r)$ be a reduced expression of $w$. We have
\[\Delta^+\cap w\Delta^-\Seteq{\beta_k=s_{i_1}\cdots s_{i_{k-1}}\alpha_{i_k}}{k\in [r]}.\]
We use this finite inversion set in the support conditions below.


\begin{definition}
Let $v,w\in W$ with $v\le w$. 

\begin{enumerate}
\item The full subcategory $\mathscr{C}_w$ of $R\operatorname{-gmod}$ consists of modules $M$ such that
\[
\mathbf{W}(M)
\subset
\operatorname{span}_{\mathbb{R}_{\ge 0}}
\bigl(\Delta^+ \cap w\Delta^- \bigr).
\]

\item The full subcategory $\mathscr{C}_{*,v}$ of $R\operatorname{-gmod}$ consists of modules $M$ such that
\[
\mathbf{W}^*(M)
\subset
\operatorname{span}_{\mathbb{R}_{\ge 0}}
\bigl(\Delta^+ \cap v\Delta^+ \bigr).
\]
\end{enumerate}

We define
\[
\mathscr{C}_{w,v}
:=
\mathscr{C}_w
\cap
\mathscr{C}_{*,v}.
\]
\end{definition}

\begin{proposition}[{\cite[Proposition~2.16]{kashiwara2018monoidal}}]
The categories $\mathscr C_w$, $\mathscr C_{*,v}$, and
$\mathscr C_{w,v}$ are closed under subquotients, extensions,
convolution products, and grading shifts.
\end{proposition}

We will use the following lemma throughout the remainder of the paper.

\begin{lemma}\label{lem:McircN}
Let $M,N$ be nonzero modules, and set $L=M\circ N$.

\begin{enumerate}
\item If $L\in \mathscr{C}_w$, then $M\in \mathscr{C}_w$.
\item If $L\in \mathscr{C}_{*,v}$, then $N\in \mathscr{C}_{*,v}$.
\end{enumerate}
\end{lemma}

\begin{proof}
Let $\operatorname{wt}(M)=\alpha$ and $\operatorname{wt}(N)=\beta$. 
Suppose that $e(\gamma,\alpha-\gamma)M\neq 0$. 
We claim that
\[
e(\gamma,\alpha+\beta-\gamma)L\neq 0.
\]

The shuffle basis for induction contains the identity-shuffle summand
$M\boxtimes N$. Applying the idempotent
$e(\gamma,\alpha-\gamma,\beta)$ to this summand gives
$e(\gamma,\alpha-\gamma)M\boxtimes N\ne0$.
Since this idempotent is a summand of
$e(\gamma,\alpha+\beta-\gamma)$, we obtain
\[
e(\gamma,\alpha+\beta-\gamma)L\neq 0.
\]

This proves that $\mathbf{W}(M)\subset \mathbf{W}(L)$. 
Therefore, if $L\in \mathscr{C}_w$, then
\[
\mathbf{W}(M)
\subset
\mathbf{W}(L)
\subset
\operatorname{span}_{\mathbb{R}_{\ge 0}}(\Delta^+\cap w\Delta^-),
\]
and hence $M\in \mathscr{C}_w$.

The second statement is proved similarly. Indeed, one shows that
\[
\mathbf{W}^*(N)\subset \mathbf{W}^*(L),
\]
and thus $L\in \mathscr{C}_{*,v}$ implies $N\in \mathscr{C}_{*,v}$.
\end{proof}

\begin{theorem}[{\cite{kang2018monoidal}}]
The Grothendieck ring $K(\mathscr{C}_w)$ provides a monoidal categorification of
\[
A_q(\mathfrak{n}(w))_{\KK}.
\]
In particular, every cluster monomial corresponds, up to the standard
quantum normalization, to the class of a real self-dual simple object in
$\mathscr{C}_w$. The initial monoidal seed is
\begin{equation}\label{eq:initialseed}
\mathbf{s}(\overline{w})
:=
\left(
\left\{
M\binom{w_k\varpi_{i_k}}{\varpi_{i_k}}
\right\}_{k\in [r]},
\,\Lambda_{\overline{w}},
\, B_{\overline{w}},\, J_{\mathrm{uf}}
\right),
\end{equation}
where $J_{\mathrm{uf}}$ is the mutable set of the word seed. Its quantum
variables are the normalized classes from
\eqref{eq:normalized-word-seed}.
We denote by $M_k$ the simple module $M\binom{w_k\varpi_{i_k}}{\varpi_{i_k}}$.
With the $R$-matrix degree convention used here, the commutation matrix is
\[
(\Lambda_{\overline w})_{pq}=-\Lambda(M_p,M_q)
\qquad(p,q\in[r]).
\]
\end{theorem}

For the category $\mathscr{C}_{w,v}$, we have the following result.

\begin{theorem}[{\cite[Proposition~4.8]{kashiwara2018monoidal},
\cite[Theorem~4.4]{kashiwara2023localizations}}]\label{thm:inCwv}
Let $v\le w$ and let $\overline{w}=(i_1\cdots i_r)$ be a reduced expression of $w$. Then:

\begin{enumerate}
\item The determinantial modules 
\[
M\binom{w_k\varpi_{i_k}}{\overline v_k\varpi_{i_k}}
\]
belong to $\mathscr{C}_{w,v}$.

\item The determinantial modules $M\binom{w\varpi_j}{v\varpi_j}$
form a real commuting family of graded central objects in
$\mathscr C_{w,v}$. In particular, the localized category
\[
\widetilde{\mathscr{C}}_{w,v}
:=
\mathscr{C}_{w,v}
\left[
M\binom{w\varpi_j}{v\varpi_j}^{-1}
\;\middle|\;
j\in I
\right]
\]
is defined, and its graded Grothendieck ring is the corresponding
Ore localization of $K(\mathscr C_{w,v})$. In finite ADE type,
specialization and complexification give
\[
K(\widetilde{\mathscr C}_{w,v})
\otimes_{\mathbb Z[q^{\pm1/2}]}\mathbb C
\simeq \mathbb C[\mathring{\mathcal B}_{v,w}],
\qquad q^{1/2}\longmapsto1,
\]
by \cite[Theorem~2.12]{leclerc2016cluster} and
\cite[Theorem~2.20(ii)(c) and Remarks~2.19, 2.21(ii)]{kashiwara2018monoidal}.
\end{enumerate}
\end{theorem}

To determine which determinantial modules do not belong to \(\mathscr{C}_{w,v}\), we shall use the following lemma.

\begin{lemma}\label{lem:vvsiw}
Let \(v\leq vs_i\leq w\). Then
\[
M\binom{w\varpi_i}{v\varpi_i}\notin \mathscr{C}_{w,vs_i}.
\]
\end{lemma}

\begin{proof}
By \cite[Theorem~10.3.1]{kang2018monoidal}, we have
\[
M\binom{w\varpi_i}{v\varpi_i}
=
M\binom{w\varpi_i}{vs_i\varpi_i}
\bigtriangledown
M\binom{vs_i\varpi_i}{v\varpi_i}.
\]
Here $N\bigtriangledown L$ refers to the head of the module $N\circ L$.
By \cite[Theorem~4.5(ii)]{kashiwara2018monoidal}, the two factors belong
to $\mathscr C_{w,vs_i}$ and $\mathscr C_{vs_i,v}$, respectively.
Their defining support conditions therefore give
\[
\mathbf{W}^*\!\left(M\binom{w\varpi_i}{vs_i\varpi_i}\right)
\subset
\operatorname{span}_{\RR_{\geq 0}}(\Delta^+\cap vs_i\Delta^+),
\]
and
\[
\mathbf{W}\!\left(M\binom{vs_i\varpi_i}{v\varpi_i}\right)
\subset
\operatorname{span}_{\RR_{\geq 0}}(\Delta^+\cap vs_i\Delta^-).
\]
In particular,
\[
\mathbf{W}^*\!\left(M\binom{w\varpi_i}{vs_i\varpi_i}\right)
\cap
\mathbf{W}\!\left(M\binom{vs_i\varpi_i}{v\varpi_i}\right)
=0.
\]
It follows from \cite[Lemma~2.8]{tingley2016mirkovic} that
\[
e(vs_i\varpi_i-w\varpi_i,\,v\alpha_i)
\Bigl(
M\binom{w\varpi_i}{vs_i\varpi_i}
\circ
M\binom{vs_i\varpi_i}{v\varpi_i}
\Bigr)
=
M\binom{w\varpi_i}{vs_i\varpi_i}
\boxtimes
M\binom{vs_i\varpi_i}{v\varpi_i}.
\]

Now let \(M\) be a proper submodule of
\[
M\binom{w\varpi_i}{vs_i\varpi_i}
\circ
M\binom{vs_i\varpi_i}{v\varpi_i}.
\]
We claim that
\[
e(vs_i\varpi_i-w\varpi_i,\,v\alpha_i)M=0.
\]
Indeed, if this were not the case, then
\[
e(vs_i\varpi_i-w\varpi_i,\,v\alpha_i)M
=
M\binom{w\varpi_i}{vs_i\varpi_i}
\boxtimes
M\binom{vs_i\varpi_i}{v\varpi_i},
\]
since the latter is simple. This would imply
\[
M=
M\binom{w\varpi_i}{vs_i\varpi_i}
\circ
M\binom{vs_i\varpi_i}{v\varpi_i},
\]
contradicting the assumption that \(M\) is proper. Hence
\[
e(vs_i\varpi_i-w\varpi_i,\,v\alpha_i)
\Bigl(
M\binom{w\varpi_i}{vs_i\varpi_i}
\bigtriangledown
M\binom{vs_i\varpi_i}{v\varpi_i}
\Bigr)
=
M\binom{w\varpi_i}{vs_i\varpi_i}
\boxtimes
M\binom{vs_i\varpi_i}{v\varpi_i}.
\]
It follows that
\[
v\alpha_i\in
\mathbf{W}^*\!\left(M\binom{w\varpi_i}{v\varpi_i}\right).
\]
Since
\[
(vs_i)^{-1}(v\alpha_i)=-\alpha_i\notin Q^+,
\]
whereas $(vs_i)^{-1}$ sends the cone generated by
$\Delta^+\cap vs_i\Delta^+$ into the positive root cone,
we conclude that
\[
M\binom{w\varpi_i}{v\varpi_i}\notin \mathscr{C}_{w,vs_i}.
\]
\end{proof}

The following proposition will be used repeatedly in the rest of the paper.

\begin{proposition}[{\cite[Theorem~4.10]{kashiwara2019laurent}}]\label{pro:commutenotone}
Let $\mathscr T=(N_i)_{i\in[r]}$ be a monoidal seed of
$\mathscr C_w$, let $k$ be a mutable vertex, and let $M$ be a
simple module in $\mathscr C_w$. If $M$ strongly commutes with
$N_j$ for all $j\ne k$, then $M$ is a cluster monomial with respect
to either $\mathscr T$ or $\mu_k\mathscr T$.
\end{proposition}

\begin{definition}
Let \(M\) and \(N\) be simple modules. We say that \(M\) is a \emph{factor} of \(N\) if there exists a simple module \(L\) such that
\[
N \cong q^a M\circ L\qquad\text{for some }a\in\mathbb Z,
\]
and \(M\) and \(L\) strongly commute. In this case, we write \(M\mid N\).
\end{definition}

\begin{lemma}\label{lem:factor}
If a simple module \(M\) is a factor of a simple module \(N\), then
\[
N\in \mathscr{C}_{w,v}\quad \Longrightarrow\quad M\in \mathscr{C}_{w,v}.
\]
Equivalently,
\[
M\notin \mathscr{C}_{w,v}\quad \Longrightarrow\quad N\notin \mathscr{C}_{w,v}.
\]
\end{lemma}

\begin{proof}
By Lemma~\ref{lem:McircN}, the relation
\[
N=M\circ L\in \mathscr{C}_w
\]
implies that \(M\in \mathscr{C}_w\).

Since \(M\) and \(L\) strongly commute, we have
\[
M\circ L \cong L\circ M.
\]
Hence, if \(N\in \mathscr{C}_{*,v}\), then Lemma~\ref{lem:McircN} again implies that \(M\in \mathscr{C}_{*,v}\).

Therefore,
\[
M\in \mathscr{C}_w\cap \mathscr{C}_{*,v}
=
\mathscr{C}_{w,v}.
\]
\end{proof}

\subsection{Mutation operators on the initial seed}
For $v\leq w$, let $\overline{v}=(i_{p_1}\cdots i_{p_m})$ be the leftmost subexpression of a reduced expression $\overline{w}=(i_1\cdots i_r)$ of $w$. For each $l\in [m]$, consider the seed
\[
\widetilde{\mathbf{s}}(\overline{v}^{\,l},\overline{w}),
\]
defined in Definition~\ref{def:subwordmutation}. By construction, this seed is obtained from the initial seed $\mathbf{s}(\overline{w})$ in \eqref{eq:initialseed} by a sequence of mutations. We denote by $M_s^l$ the cluster variable at the vertex $s$ of the seed $\widetilde{\mathbf{s}}(\overline{v}^{\,l},\overline{w})$. If $s=(i,d)$, we also write $M_{(i,d)}^l$.

We begin with the seed $\widetilde{\mathbf{s}}(\overline{v}^{\,1},\overline{w})$. Recall that
\(
p_1=(i,d_1), \qquad p_1^{\max}=(i,n_i).
\)
The mutation sequence $\widetilde{\mu}_1$ consists of the mutations $\mu_{(i,k)}$ for
\(
(i,k)\in \bigl(I^1_{i,1}\setminus\{p_1\}\bigr)[-1].
\)
For simplicity, we write $\mu_{(i,k)}$ simply as $\mu_k$. Then
\(
\widetilde{\mu}_1=\mu_{n_i-1}\circ \cdots \circ \mu_{b_1+1}.
\)

\begin{lemma}\label{lem:mutation1}
For every $(i,k)\in \bigl(I^1_{i,1}\setminus\{p_1\}\bigr)[-1]$, one has
\begin{equation}\label{eq:M1ik}
M^1_{(i,k)}
\mid
M\binom{w_{(i,k+1)}\varpi_i}{s_i\varpi_i}
\qquad\text{with multiplicity }1.
\end{equation}
Moreover, every determinantial module in
\[
\mathscr{T}_1\setminus
\left\{
M\binom{w_{p_1}\varpi_{i_{p_1}}}{s_i\varpi_{i_{p_1}}}
\right\}
\]
is a cluster monomial in the seed $\mathbf{s}(\overline{v}^{\,1},\overline{w})$.
\end{lemma}

\begin{proof}
If $n_i=d_1$, the mutation sequence is empty. All the retained
determinantial modules coincide with their initial counterparts by
Lemma~\ref{lem:l=l-1}; thus the divisibility assertion is vacuous and
the cluster-monomial assertion follows directly. Assume henceforth
that $n_i>d_1$.

We argue inductively along $\widetilde\mu_1$, keeping track of two
statements simultaneously: every module in the current family
$\mathscr T_{1,t}$ is a cluster monomial in the current ambient seed,
and each seed variable occurs with multiplicity one in the module
assigned to it by $\Phi^{1,t}$. At each step the strong commutation
and multiplicity calculations use Lemma~\ref{lem:d-additivity}.

First consider the initial step, corresponding to the family $\mathscr{T}_{1,1}$. One checks directly that
\[
Y_{1,1}=\{(i,d_1+1)\},
\qquad
T_{1,1}=[r]\setminus\{(i,d_1)\}.
\]
By Proposition~\ref{pro:Tlkcommute}, the module
\[
M\binom{w_{(i,d_1+1)}\varpi_i}{s_i\varpi_i}
\]
strongly commutes with $M_s$ for all $s\neq (i,d_1)$. Hence, by Proposition~\ref{pro:commutenotone}, it is a cluster monomial with respect to either the seed $\mathbf{s}(\overline{w})$ or the seed $\mu_{d_1}\mathbf{s}(\overline{w})$.

On the other hand, Proposition~\ref{pro:Tlkcommute} also shows that
\[
M\binom{w_{(i,d_1+1)}\varpi_i}{s_i\varpi_i}
\]
does not strongly commute with $M_{(i,d_1)}$. Therefore it cannot be a cluster monomial in the seed $\mathbf{s}(\overline{w})$, and hence must be a cluster monomial in the mutated seed $\mu_{d_1}\mathbf{s}(\overline{w})$. It follows that $M^1_{(i,d_1)}$ occurs as a factor.

Let $n$ be the multiplicity of $M^1_{(i,d_1)}$ in this factorization. Then
\[
\mathfrak{d}\left(
M\binom{w_{(i,d_1+1)}\varpi_i}{s_i\varpi_i},
M\binom{w_{(i,d_1)}\varpi_i}{\varpi_i}
\right)
=
n\,\mathfrak{d}\left(
M^1_{(i,d_1)},
M\binom{w_{(i,d_1)}\varpi_i}{\varpi_i}
\right)
\ge n.
\]
By Proposition~\ref{pro:Tlkcommute}, we conclude that $n=1$. Thus
\[
M^1_{(i,d_1)}
\mid
M\binom{w_{(i,d_1+1)}\varpi_i}{s_i\varpi_i}
\]
with multiplicity one.

Moreover, no other determinantial module in $\mathscr{T}_{1,1}$ can contain $M_{(i,d_1)}$ as a factor; otherwise the modules in $\mathscr{T}_{1,1}$ would fail to commute pairwise. Its initial factorization therefore involves only unchanged seed variables. Thus every determinantial module in $\mathscr{T}_{1,1}$ is a cluster monomial, with nonnegative exponents, in $\mu_{d_1}\mathbf{s}(\overline{w})$. Each cluster variable $X_{(j,q)}$ of this seed occurs with multiplicity one in the determinantial module indexed by $\Phi_j^{1,1}(j,q)$.

We now proceed by induction. Suppose that, for some $t\ge 2$, the statement has been proved up to step $t-1$, and that every module in $\mathscr{T}_{1,t-1}$ is a cluster monomial, with nonnegative exponents, in the seed
\[
\mu_{d_1+t-2}\cdots\mu_{d_1}\mathbf{s}(\overline{w}).
\]
We also assume the multiplicity-one divisibility statement indexed
by $\Phi^{1,t-1}$; this is an additional induction invariant, rather
than a reformulation of the cluster-monomial statement.

Consider the family $\mathscr{T}_{1,t-1}$. By \eqref{eq:Tlkrecursion}, we have
\[
\mathscr{T}_{1,t}\setminus
\left\{
M\binom{w_{(i,d_1+t)}\varpi_i}{s_i\varpi_i}
\right\}
=
\mathscr{T}_{1,t-1}\setminus\{M_{(i,d_1+t-1)}\}.
\]
Hence, by the induction hypothesis, every cluster variable of
\[
\mu_{d_1+t-2}\cdots\mu_{d_1}\mathbf{s}(\overline{w})
\]
other than $M_{(i,d_1+t-1)}$ already appears as a factor of a determinantial module in $\mathscr{T}_{1,t}\setminus\{M\binom{w_{(i,d_1+t)}\varpi_i}{s_i\varpi_i}\}$.

Now apply the same argument as in the initial step. Using Proposition~\ref{pro:Tlkcommute}, Proposition~\ref{pro:commutenotone}, and the fact that $M_{(i,d_1+t-1)}$ is a cluster variable of the seed
\[
\mu_{d_1+t-2}\cdots\mu_{d_1}\mathbf{s}(\overline{w}),
\]
we obtain that
\[
M\binom{w_{(i,d_1+t)}\varpi_i}{s_i\varpi_i}
\]
is a cluster monomial in the seed
\[
\mu_{d_1+t-1}\cdots\mu_{d_1}\mathbf{s}(\overline{w}),
\]
and that $M^1_{(i,d_1+t-1)}$ appears in it as a factor with multiplicity one.

By the definition of $\Phi_j^{1,t}$ and equations~\eqref{eq:Tlk-1Ylk} and~\eqref{eq:Phil-1Phil}, it follows that every cluster variable $X_{(j,s)}$ in the seed
\[
\mu_{d_1+t-1}\cdots\mu_{d_1}\mathbf{s}(\overline{w})
\]
is a multiplicity-one factor of a determinantial module in $\mathscr{T}_{1,t}$ indexed by $\Phi_j^{1,t}(j,s)$.

Furthermore, no other determinantial module in $\mathscr{T}_{1,t}$ can contain $M_{(i,d_1+t-1)}$ as a factor; otherwise the modules in $\mathscr{T}_{1,t}$ would fail to commute pairwise. Consequently, every surviving old factorization uses only unchanged seed variables and retains all its exponents. Together with the new minor, this shows that every module in $\mathscr{T}_{1,t}$ is a cluster monomial, with nonnegative exponents, in the seed
\[
\mu_{d_1+t-1}\cdots\mu_{d_1}\mathbf{s}(\overline{w}).
\]

Iterating this argument up to $t=n_i-d_1$, we conclude that every module in $\mathscr{T}_{1,n_i-d_1}$ is a cluster monomial in the seed
\[
\widetilde{\mu}_1\mathbf{s}(\overline{w}).
\]
By \eqref{eq:Tlni-d}, for every $(j,k)\neq (i,n_i)$, the cluster variable $M^1_{(j,k)}$ of $\widetilde{\mu}_1\mathbf{s}(\overline{w})$ is a multiplicity-one factor of the determinantial module in $\mathscr{T}_1$ indexed by $\Phi_j^1(j,k)$. Moreover, every determinantial module in
\[
\mathscr{T}_1\setminus
\left\{
M\binom{w_{p_1}\varpi_{i_{p_1}}}{s_i\varpi_{i_{p_1}}}
\right\}
\]
is a cluster monomial in the seed $\widetilde{\mu}_1\mathbf{s}(\overline{w})$.

Finally, Lemma~\ref{lem:vvsiw} shows that
\[
M\binom{w\varpi_{i_{p_1}}}{\varpi_{i_{p_1}}}\notin \mathscr{C}_{w,v^1}.
\]
Hence this module cannot occur as a factor of any determinantial module in $\mathscr{T}_1$. Therefore every determinantial module in
\[
\mathscr{T}_1\setminus
\left\{
M\binom{w_{p_1}\varpi_{i_{p_1}}}{s_i\varpi_{i_{p_1}}}
\right\}
\]
is already a cluster monomial in the seed $\mathbf{s}(\overline{v}^1,\overline{w})$. This proves the lemma.
\end{proof}

\begin{theorem}\label{thm:dividminors}
For $l\in [m]$ and $j\in I$, let $(j,k)\in I_{j,\le l}$, and write
\[
\Phi_j^l((j,k))=(j,k+s),
\]
where $\Phi_j^l$ is the map defined in~\eqref{eq:Phiil}. Then
\begin{equation}\label{eq:Mjsl}
M_{(j,k)}^l
\mid
M\!\binom{w_{(j,k+s)}\varpi_j}
{\overline{v}^{\,l}_{(j,k+s)}\varpi_j}
\qquad\text{with multiplicity }1.
\end{equation}
Moreover, for every $q\notin \{p_1,\dots,p_l\}$, the determinantial module
\[
M\!\binom{w_q\varpi_{i_q}}{\overline{v}^{\,l}_q\varpi_{i_q}}
\]
is a cluster monomial in the seed $\mathbf{s}(\overline{v}^{\,l},\overline{w})$.
\end{theorem}

\begin{proof}
We prove the two assertions of the theorem together with the
exclusion statement \eqref{eq:notincwv-new} by simultaneous induction
on $l$. Thus, at stage $l$, all three statements are available at
stage $l-1$. Within each stage we maintain both the complete
cluster-monomial factorizations and the divisibility statements
indexed by $\Phi^{l,k}$.

\medskip
\noindent
\textit{Step 1: the case $l=1$.}
Let $i=i_{p_1}$. Then
\[
I_i=I_{i,0}^1\sqcup I_{i,1}^1,
\qquad
I_j=I_{j,0}^1 \quad (j\neq i),
\]
and, since $p_1=(i,b_1+1)$, we have $t_1=b_1+1$. More precisely,
\[
I_{i,0}^1=\{(i,k)\in I_i\mid (i,k)<p_1\},
\qquad
I_{j,0}^1=I_j \quad (j\neq i).
\]

For every $q\in I_{i,0}^1\cup I_{j,0}^1$, we have
\[
M_q^1=M_q^0
=
M\!\binom{w_q\varpi_{i_q}}{\overline{v}^{\,0}_q\varpi_{i_q}},
\]
and Lemma~\ref{lem:l=l-1} gives
\[
M_q^1
=
M\!\binom{w_q\varpi_{i_q}}{\overline{v}^{\,1}_q\varpi_{i_q}}.
\]

Next,
\[
I_{i,1}^1\setminus\{p_1\}
=
\{(i,k)\in I_i\mid b_1+2\le k\le n_i\}.
\]
Hence Lemma~\ref{lem:mutation1} yields
\[
M_{(i,k)}^1
\mid
M\!\binom{w_{(i,k+1)}\varpi_i}{\overline{v}^{\,1}_{(i,k+1)}\varpi_i}
\qquad
\text{for }k\in [b_1+1,n_i-1].
\]
The cluster-monomial assertion follows from
Lemma~\ref{lem:mutation1}. The only deleted variable is
$M^1_{(i,n_i)}=M\binom{w\varpi_i}{\varpi_i}$, which is excluded from
$\mathscr C_{w,v^1}$ by Lemma~\ref{lem:vvsiw}. This establishes all
the induction statements for $l=1$.

\medskip
\noindent
\textit{Step 2: the induction step and the complete factorizations.}
Assume that all the induction statements hold for $l-1$, and set
$i=i_{p_l}$.

We first treat the case $j\neq i$. By the induction hypothesis,
\begin{equation}\label{eq:induction-jneqi}
M_{(j,k)}^{\,l-1}
\mid
M\!\binom{w_{(j,k+s)}\varpi_j}
{\overline{v}^{\,l-1}_{(j,k+s)}\varpi_j}
\qquad
\text{for all }(j,k+s)\in I_{j,s}^{\,l-1}\setminus\{(j,t_s)\}.
\end{equation}
Since $j\neq i$, we have $\alpha(j,l)=\alpha(j,l-1)$, and hence
\[
I_{j,s}^{\,l}=I_{j,s}^{\,l-1}
\qquad
\text{for all } s\in [0,\alpha(j,l)].
\]
Together with Lemma~\ref{lem:l=l-1}, this implies
\[
M_{(j,k)}^{\,l}=M_{(j,k)}^{\,l-1}
\mid
M\!\binom{w_{(j,k+s)}\varpi_j}
{\overline{v}^{\,l}_{(j,k+s)}\varpi_j}
\qquad
\text{for all } (j,k+s)\in I_{j,s}^{\,l}\setminus\{(j,t_s)\}.
\]

Now consider the color $i=i_{p_l}$. Note that
\[
d_l=b_l+\alpha(i,l)
\]
and
\[
\bigl(I_{i,\alpha(i,l)}^l\setminus\{p_l\}\bigr)[-\alpha(i,l)]
=
\{(i,b_l+1),\dots,(i,n_i-\alpha(i,l))\}.
\]
Thus the mutation sequence $\widetilde{\mu}_l$ is precisely the sequence of mutations at these vertices.

For $q\in \bigl(I_{i,s}^l\setminus\{(i,t_s)\}\bigr)[-s]$ with $s<\alpha(i,l)$, we therefore have
\[
M_q^l=M_q^{l-1}.
\]
Moreover,
\[
\overline{v}^{\,l-1}_{(i,k+s)}=\overline{v}^{\,l}_{(i,k+s)}
\qquad
\text{for } (i,k+s)\in I_{i,s}^l\setminus\{(i,t_s)\},\ s<\alpha(i,l),
\]
and \(I_{i,s}^l\subset I_{i,s}^{\,l-1}\) by \eqref{eq:Iisl}. Hence the induction hypothesis and Lemma~\ref{lem:l=l-1} give
\[
M_{(i,k)}^l
=
M_{(i,k)}^{\,l-1}
\mid
M\!\binom{w_{(i,k+s)}\varpi_i}
{\overline{v}^{\,l}_{(i,k+s)}\varpi_i}
\qquad
\text{for all } (i,k+s)\in I_{i,s}^l\setminus\{(i,t_s)\},\ s<\alpha(i,l).
\]

If $d_l=n_i$, the new layer and the mutation sequence are empty.
For every retained index $q\notin\{p_1,\ldots,p_l\}$, the relevant
minor is unchanged by Lemma~\ref{lem:l=l-1}, so its complete ambient
cluster-monomial factorization is supplied by induction. The
remaining divisibility statements have already been proved above;
we proceed directly to Step~3. In what follows, assume $d_l<n_i$.

It remains to analyze the new layer corresponding to
$s=\alpha(i,l)$. By \eqref{eq:Tlkrecursion},
\begin{equation}\label{eq:Tl1-rec}
\mathscr{T}_{l,1}\setminus
\left\{
M\!\binom{w_{(i,d_l+1)}\varpi_i}
{\overline{v}^{\,l}_{(i,d_l+1)}\varpi_i}
\right\}
=
\mathscr{T}_{l-1}\setminus
\left\{
M\!\binom{w_{(i,d_l)}\varpi_i}
{\overline{v}^{\,l-1}_{(i,d_l)}\varpi_i}
\right\}.
\end{equation}
Since
\[
(i,d_l)\in I_{i,\alpha(i,l-1)}^{\,l-1},
\qquad
\alpha(i,l-1)=\alpha(i,l)-1,
\]
the induction hypothesis shows that
\[
M_{(i,b_l+1)}^{\,l-1}
\mid
M\!\binom{w_{(i,d_l)}\varpi_i}
{\overline{v}^{\,l-1}_{(i,d_l)}\varpi_i}.
\]

On the other hand, every other cluster variable of
\(\widetilde{\mathbf{s}}(\overline{v}^{\,l-1},\overline{w})\)
appears as a factor of some determinantial module in the family
\[
\mathscr{T}_{l-1}\cup
\left\{
M\!\binom{w\varpi_j}{v^s\varpi_j}
\;\middle|\;
j\in I,\ 0\leq s<l-1
\right\}.
\]
Indeed, let us consider a cluster variable \(M_{(j,k)}^{\,l-1}\) such that
\[
k+\alpha(j,l-1)>n_j.
\]
By the induction hypothesis, this module is not a factor of any determinantial module in \(\mathscr{T}_{l-1}\). For such a pair \((j,k)\), there exists an integer \(0\leq s<l-1\) such that
\[
k+\alpha(j,s)=n_j.
\]
Since the cluster variable is not mutated after it is deleted at
stage $s+1$, we have
\[
M_{(j,k)}^{\,l-1}=M_{(j,k)}^{\,s}.
\]
Applying the induction hypothesis again (or the initial-seed identity
\eqref{eq:initialseed} when $s=0$), we see that
\(M_{(j,k)}^{\,s}\) appears as a factor of
\[
M\binom{w\varpi_j}{v^s\varpi_j}.
\]

Consequently, every cluster variable \(M_q^{\,l-1}\) of
\(\widetilde{\mathbf{s}}(\overline{v}^{\,l-1},\overline{w})\)
is a factor of some determinantial module in the family
\[
\mathscr{T}_{l-1}
\cup
\left\{
M\binom{w\varpi_j}{v^s\varpi_j}
\;\middle|\;
j\in I,\ 0\leq s<l-1
\right\}.
\]
Hence, by \eqref{eq:Tl1-rec}, the only cluster variable of
\(\widetilde{\mathbf{s}}(\overline{v}^{\,l-1},\overline{w})\)
which may fail to be represented in
\[
\mathscr{T}_{l,1}\cup
\left\{
M\!\binom{w\varpi_j}{v^s\varpi_j}
\;\middle|\;
j\in I,\ 0\leq s<l-1
\right\}
\]
is \(M_{(i,b_l+1)}^{\,l-1}\).

By Proposition~\ref{lem:comwvs}, the new determinantial module
strongly commutes with the displayed family. Strong commutation
passes to its real convolution factors by
Lemma~\ref{lem:d-additivity}. Consequently,
Proposition~\ref{pro:commutenotone} implies that
\[
M\!\binom{w_{(i,d_l+1)}\varpi_i}
{\overline{v}^{\,l}_{(i,d_l+1)}\varpi_i}
\]
is a cluster monomial in the cluster variables of either
\[
\widetilde{\mathbf{s}}(\overline{v}^{\,l-1},\overline{w})
\qquad\text{or}\qquad
\mu_{(i,b_l+1)}\widetilde{\mathbf{s}}(\overline{v}^{\,l-1},\overline{w}).
\]
If it were a cluster monomial in
\(\widetilde{\mathbf{s}}(\overline{v}^{\,l-1},\overline{w})\),
then, by the induction hypothesis,
\[
M\!\binom{w_{(i,d_l)}\varpi_i}
{\overline{v}^{\,l-1}_{(i,d_l)}\varpi_i}
\]
would also be a cluster monomial in that seed. This would imply
\[
\mathfrak d\!\left(
M\!\binom{w_{(i,d_l+1)}\varpi_i}
{\overline{v}^{\,l}_{(i,d_l+1)}\varpi_i},
\,
M\!\binom{w_{(i,d_l)}\varpi_i}
{\overline{v}^{\,l-1}_{(i,d_l)}\varpi_i}
\right)=0,
\]
contrary to \eqref{eq:dMlml-1}. Therefore
\[
M_{(i,b_l+1)}^{\,l}
\mid
M\!\binom{w_{(i,d_l+1)}\varpi_i}
{\overline{v}^{\,l}_{(i,d_l+1)}\varpi_i},
\]
and the multiplicity-one assertion follows from
\eqref{eq:dMlml-1} and Lemma~\ref{lem:d-additivity}.

We spell out why the complete factorization invariant also survives
each inner step. Denote the old pivot by $X$, its mutation by $X'$, and
the replaced and new determinantial modules by $Q$ and $P$, respectively.
The induction hypothesis gives $Q\simeq X\circ A$, where $A$ is a
convolution monomial in the other current seed variables. The coverage
argument above shows that $P$ strongly commutes with those variables.
Thus Lemma~\ref{lem:d-additivity} and \eqref{eq:dMlml-1} give
\[
1=\mathfrak d(P,Q)=\mathfrak d(P,X).
\]
Every surviving old determinantial module whose factorization is part
of the induction invariant strongly commutes with $P$. Its exponent of
$X$ must therefore be zero, by Lemma~\ref{lem:d-additivity}.
Its entire nonnegative factorization consequently remains unchanged
after replacing $X$ by $X'$, including the multiplicity of its assigned
seed variable. The assigned minor for an unchanged variable survives
by \eqref{eq:Tlkrecursion} and the definition of $\Phi^{l,k}$.
Finally, if $P\simeq (X')^{\circ n}\circ B$ is its factorization in the
mutated seed, then $n\geq1$ and
\[
1=\mathfrak d(P,X)=n\,\mathfrak d(X',X),
\]
so $n=1$. This preserves both the complete cluster-monomial
factorizations and the assigned multiplicity-one divisibility statements.

By \eqref{eq:Phil-1Phil}, every cluster variable $X_{(i,s)}$ with
$(i,s)\in I_{i,\le l-1}$ in
\[
\mu_{(i,b_l+1)}\widetilde{\mathbf{s}}(\overline{v}^{\,l-1},\overline{w})
\]
is a multiplicity-one factor of the determinantial module in \(\mathscr{T}_{l,1}\) indexed by \(\Phi_i^{l,1}(i,s)\).

Repeating the same argument for the successive mutations
$\mu_{(i,b_l+2)},\ldots,\mu_{(i,n_i-\alpha(i,l))}$, we obtain, for
$1\le k\le n_i-d_l$,
\begin{equation}\label{eq:Mliblk-new}
M_{(i,b_l+k)}^{\,l}
\mid
M\!\binom{w_{(i,d_l+k)}\varpi_i}
{\overline{v}^{\,l}_{(i,d_l+k)}\varpi_i},
\end{equation}
and that
\[
M\!\binom{w_{(i,d_l+k)}\varpi_i}
{\overline{v}^{\,l}_{(i,d_l+k)}\varpi_i}
\]
is a cluster monomial in the seed
\[
\mu_{(i,b_l+k)}\cdots\mu_{(i,b_l+1)}
\widetilde{\mathbf{s}}(\overline{v}^{\,l-1},\overline{w}).
\]

Furthermore, no determinantial module in
\[
\mathscr{T}_{l,k}\setminus
\left\{
M\!\binom{w_{p_t}\varpi_{i_{p_t}}}
{\overline{v}^{\,l-1}_{p_t}\varpi_{i_{p_t}}}
\;\middle|\;
t\in [l-1]
\right\}
\]
can contain \(M_{(i,b_l+s)}^{\,l-1}\) as a factor for any \(1\le s\le k\). Otherwise it would fail to strongly commute with
\[
M\!\binom{w_{(i,d_l+s)}\varpi_i}
{\overline{v}^{\,l}_{(i,d_l+s)}\varpi_i},
\]
which also belongs to \(\mathscr{T}_{l,k}\), contradicting Proposition~\ref{pro:Tlkcommute}. Hence every such determinantial module is a cluster monomial in the seed
\[
\mu_{(i,b_l+k)}\cdots\mu_{(i,b_l+1)}
\widetilde{\mathbf{s}}(\overline{v}^{\,l-1},\overline{w}).
\]

Taking \(k=n_i-d_l\) and using \eqref{eq:Tlni-d}, we deduce that
\[
M\!\binom{w_q\varpi_{i_q}}{\overline{v}^{\,l}_q\varpi_{i_q}}
\]
is a cluster monomial in \(\widetilde{\mathbf{s}}(\overline{v}^{\,l},\overline{w})\) for every \(q\notin \{p_1,\dots,p_l\}\).

Finally, since \(\Phi_i^{l,n_i-d_l}=\Phi_i^l\) on \(I_{i,\le l}\) and \(\Phi_j^l=\Phi_j^{l-1}\) for \(j\neq i\), equations \eqref{eq:Mliblk-new} and \eqref{eq:Tlni-d} show that, for every \((j,k)\in I_{j,\le l}\), the determinantial module in \(\mathscr{T}_l\) indexed by \(\Phi_j^l(j,k)\) contains \(M_{(j,k)}^l\) as a multiplicity-one factor. This proves \eqref{eq:Mjsl}.

\medskip
\noindent
\textit{Step 3: exclusion from \(\mathscr{C}_{w,v^l}\).}
We next prove that
\begin{equation}\label{eq:notincwv-new}
M_{(i,k)}^l\notin \mathscr{C}_{w,v^l}
\qquad
\text{whenever } k+\alpha(i,l)>n_i.
\end{equation}

The case $l=1$ was included in Step~1: the only deleted variable is
$M_{(i,n_i)}^1$, and
\[
M_{(i,n_i)}^1=M\!\binom{w\varpi_i}{\varpi_i}.
\]
Hence \eqref{eq:notincwv-new} follows from Lemma~\ref{lem:vvsiw}.

Use the simultaneous induction hypothesis at $l-1$, and set
$i=i_{p_l}$. If $j\ne i$, then $\alpha(j,l)=\alpha(j,l-1)$, so
\[
M_{(j,k)}^l=M_{(j,k)}^{\,l-1}
\qquad
\text{whenever } k+\alpha(j,l)>n_j.
\]
Similarly, if \(j=i\) and \(k+\alpha(i,l)>n_i\) with \(k>n_i-\alpha(i,l)+1\), then
\[
M_{(i,k)}^l=M_{(i,k)}^{\,l-1},
\]
and hence these variables do not belong to \(\mathscr{C}_{w,v^l}\), since
\[
\mathscr{C}_{w,v^l}\subset \mathscr{C}_{w,v^{l-1}}.
\]

It remains to treat the boundary term
\[
M_{(i,n_i-\alpha(i,l)+1)}^l=M_{(i,n_i-\alpha(i,l)+1)}^{l-1}.
\]
By the divisibility assertion at stage $l-1$, this module is a factor of
\[
M\!\binom{w\varpi_i}
{\overline{v}^{\,l-1}_{(i,n_i)}\varpi_i}.
\]
Note that the determinantial module
\[
M\!\binom{w\varpi_i}{\overline{v}^{\,l-1}_{(i,n_i)}\varpi_i}
\]
belongs to \(\mathscr{T}_{l,n_i-d_l}\), and therefore is a cluster monomial in the seed
\(\widetilde{\mathbf{s}}(\overline{v}^{\,l},\overline{w})\).
Since
\[
M\!\binom{w\varpi_i}{\overline{v}^{\,l-1}_{(i,n_i)}\varpi_i}\in \mathscr{C}_{w,v^{l-1}},
\]
the induction hypothesis \eqref{eq:notincwv-new} implies that all of its factors are of the form \(M^l_{(j,k)}\) with \(k+\alpha(j,l-1)\le n_j\) for every \(j\in I\).

On the other hand, for every \((j,k)\) with \(k+\alpha(j,l)\le n_j\), the cluster variable \(M_{(j,k)}^l\) is a factor of some determinantial module
\[
M\!\binom{w_t\varpi_{i_t}}{\overline{v}^{\,l}_t\varpi_{i_t}}
\in \mathscr{C}_{w,v^l},
\qquad
t=\Phi^l(j,k).
\]
Hence all such \(M_{(j,k)}^l\) belong to \(\mathscr{C}_{w,v^l}\).

The factors of $v^{l-1}$ following the last occurrence $(i,n_i)$
have colors different from $i$ and hence fix $\varpi_i$. Therefore
$\overline v^{\,l-1}_{(i,n_i)}\varpi_i=v^{l-1}\varpi_i$.
Since $v^l=v^{l-1}s_i$, Lemma~\ref{lem:vvsiw} shows that
\[
M\!\binom{w\varpi_i}
{\overline{v}^{\,l-1}_{(i,n_i)}\varpi_i}
\notin \mathscr{C}_{w,v^l}.
\]
Therefore at least one of its factors must lie outside \(\mathscr{C}_{w,v^l}\), and the preceding discussion shows that the only possibility is
\[
M_{(i,n_i-\alpha(i,l)+1)}^l.
\]
Thus
\[
M_{(i,n_i-\alpha(i,l)+1)}^l\notin \mathscr{C}_{w,v^l},
\]
which completes the proof of \eqref{eq:notincwv-new}.

\medskip
\noindent
\textit{Step 4: conclusion.}
By Theorem~\ref{thm:inCwv}, we have
\[
M\!\binom{w_q\varpi_{i_q}}{\overline{v}^{\,l}_q\varpi_{i_q}}
\in \mathscr{C}_{w,v^l}
\qquad
\text{for all } q\notin \{p_1,\dots,p_l\}.
\]
Since the cluster variables removed in passing from
\(\widetilde{\mathbf{s}}(\overline{v}^{\,l},\overline{w})\)
to
\(\mathbf{s}(\overline{v}^{\,l},\overline{w})\)
are precisely those excluded by \eqref{eq:notincwv-new}, it follows that every such determinantial module is in fact a cluster monomial in the seed
\[
\mathbf{s}(\overline{v}^{\,l},\overline{w}).
\]
This proves the theorem.
\end{proof}

\begin{remark}
    For any $l \in [m]$, it is straightforward to verify that the vertex set of the seed $\mathbf{s}(\overline{v}^l,\overline{w})$ coincides with $\bigsqcup_{i\in I} I_{i,\leq l}$. Hence, $\Phi^{\ell(v)}$ induces a bijection from the vertex set $J$ of the seed $\mathbf{s}(\overline{v},\overline{w})$ to the set $[r]\setminus\{p_1,\dots,p_m\}$.
\end{remark}
\begin{theorem}\label{thm:clustercat}
Every cluster variable of $\mathbf s(\overline v,\overline w)$ is
realized, with the quantum-seed normalization, by a real self-dual
simple module in $\mathscr C_{w,v}$. The corresponding family is
stable under all seed mutations, and every normalized quantum
cluster monomial is represented by a real self-dual simple module.
In particular,
\[
\overline{\mathcal A}_q(\mathbf s(\overline v,\overline w))
\subset K(\mathscr C_{w,v}),
\]
where the bar indicates that frozen variables are not inverted.
\end{theorem}

\begin{proof}
If $v=e$, the statement is the ambient monoidal categorification.
Assume $\ell(v)>0$.
    By Theorem~\ref{thm:dividminors}, each simple module
    $M^{\ell(v)}_k$ in the initial seed satisfies
    \[
        M^{\ell(v)}_k \mid 
        M\binom{w_{\Phi^{\ell(v)}(k)}\varpi_{i_{\Phi^{\ell(v)}(k)}}}{\overline{v}_{\Phi^{\ell(v)}(k)}\varpi_{i_{\Phi^{\ell(v)}(k)}}}.
    \]
    By Theorem~\ref{thm:inCwv}, for all $s\in [r]$, the determinantial module 
    $M\binom{w_{s}\varpi_{i_s}}{\overline{v}_s\varpi_{i_s}}$
    belongs to $\mathscr{C}_{w,v}$. 
    Hence, by Lemma~\ref{lem:factor}, we obtain 
    $M^{\ell(v)}_k \in \mathscr{C}_{w,v}$. 
    Therefore, all cluster variables in the initial seed 
    lie in $\mathscr{C}_{w,v}$.

Let $D$ be the deleted vertex set and $J_{\mathrm{uf}}$ the mutable
vertex set of the retained seed. By construction,
$b_{dk}=0$ for $d\in D$ and $k\in J_{\mathrm{uf}}$.
The matrix-mutation formula shows that these equalities remain
true after mutation at any $k\in J_{\mathrm{uf}}$: every correction
term in a deleted row contains $b_{dk}=0$. Thus every allowed
mutation sequence lifts to the ambient monoidal seed in
$\mathscr C_w$, with identical exchange monomials.

Suppose that the current seed modules $(X_j)$ belong to
$\mathscr C_{w,v}$, and mutate at $k$. Ambient monoidal
categorification supplies a real self-dual simple module
$X'_k\in\mathscr C_w$ and an exchange short exact sequence
\[
0\longrightarrow q^a P\longrightarrow X_k\circ X'_k
\longrightarrow q^b Q\longrightarrow0,
\]
where $P,Q$ are convolution monomials in the other retained seed
modules (after a possible common grading shift).
Since $\mathscr C_{w,v}$ is closed under convolution products and
extensions, $X_k\circ X'_k\in\mathscr C_{w,v}$.
Lemma~\ref{lem:McircN}(2) now gives
$X'_k\in\mathscr C_{*,v}$. Together with its ambient membership in
$\mathscr C_w$, this yields $X'_k\in\mathscr C_{w,v}$.

Induction proves the assertion for every seed. Pairwise strong
commutation, reality, and the exchange short exact sequences are
inherited from the ambient monoidal categorification. Normalizing
the convolution products therefore realizes all quantum cluster
monomials as asserted.
\end{proof}

\begin{example}[A nontrivial factorization in type $A_3$]
\label{ex:A3-factorization}
Take $w=w_0$ in type $A_3$, with
$\overline w=(1,2,3,2,1,2)$, and let $v=s_2$.
The leftmost subexpression selects $p_1=2$, so
$\widetilde\mu_1=\mu_4\circ\mu_2$ in positional labels.
Notice that $\overline v_1=e$, whereas $v^1=s_2$; this also
illustrates the distinction between absolute prefixes and stages.
To display the calculation, specialize at $q=1$ and write
\[
n=\begin{pmatrix}
1&a&b&c\\0&1&d&e\\0&0&1&f\\0&0&0&1
\end{pmatrix},\qquad
h:=adf-ae-bf+c,\qquad g:=be-cd.
\]
The initial variables are
\[
(x_1,x_2,x_3,x_4,x_5,x_6)=(a,ad-b,h,ae-c,c,g).
\]
The two exchange relations are
\[
(ad-b)x'_2=(ae-c)+h,\qquad
(ae-c)x'_4=afg+ch.
\]
Consequently $x'_2=f$ and $x'_4=bf-c$.
Deleting vertex $6=(2,3)$ and freezing its remaining mutable
neighbors gives
\[
(y_1,y_2,y_3,y_4,y_5)=(a,f,h,bf-c,c),\qquad
J_{\mathrm{fr}}=\{3,4,5\},\quad J_{\mathrm{uf}}=\{1,2\}.
\]
The restricted exchange matrix, with rows $1,\ldots,5$ and
columns $1,2$, is
\[
B=\begin{pmatrix}
0&0\\0&0\\-1&0\\1&1\\0&-1
\end{pmatrix}.
\]
Put
$D_t^{\mathrm{cl}}:=
D\!\binom{w_t\varpi_{i_t}}{\overline v_t\varpi_{i_t}}\big|_{q=1}$.
The map $\Phi=\Phi^1$ and all its assigned factorizations are
\[
\begin{array}{c|ccccc}
k&1&2&3&4&5\\\hline
\Phi(k)&1&4&3&6&5\\
D_{\Phi(k)}^{\mathrm{cl}}&y_1&y_1y_2&y_3&y_4&y_5
\end{array}
\]
In particular, $w_4=s_1s_2s_3s_2$ gives the ordinary minor
with row set $\{1,3\}$ and column set $\{2,4\}$:
\[
D_4^{\mathrm{cl}}
=\det\begin{pmatrix}a&c\\0&f\end{pmatrix}
=y_1y_2.
\]
Thus the variable at vertex $2$ is a multiplicity-one factor of
its assigned minor, although that minor is not itself a seed
variable. At the categorical level, $y_1$ and $y_2$ correspond
to the simple-root modules $L(1)$ and $L(3)$, respectively,
where $L(i)$ is the self-dual simple $R(\alpha_i)$-module.
These modules strongly commute, so the factorization in
Theorem~\ref{thm:dividminors} reads
\[
M\!\binom{w_4\varpi_2}{s_2\varpi_2}
\simeq L(1)\circ L(3)
\]
up to grading shift. The other rows give single determinantial
factors. This illustrates both the nontrivial factorization and
the distinguished multiplicity-one assignment.
\end{example}

\subsection{Finite ADE type}

Throughout this subsection, the Cartan datum is of finite ADE type.
We first make the comparison with Leclerc's conventions explicit.
Write $(v^{\mathrm L},w^{\mathrm L})=(v^{-1},w^{-1})$ for the
pair used in \cite{leclerc2016cluster}. The transpose
anti-automorphism gives the corresponding identification of flag
varieties, $gB_+\mapsto B_-g^{\mathrm T}$, and sends our
Richardson stratum to Leclerc's stratum for
$(v^{\mathrm L},w^{\mathrm L})$.
We write
\[
\mathcal C^{\mathrm{pre}}_{v,w}
:=\mathcal C^{\mathrm L}_{v^{-1},w^{-1}}
\]
for his Frobenius category of Dynkin preprojective-algebra
modules, with this transported indexing. It is distinct from
the quiver Hecke category $\mathscr C_{w,v}$.
When comparing exchange matrices, all seeds are written in the
same quiver convention as Section~\ref{sec:wordseed}; if a cited
geometric construction uses the opposite convention, its entire
exchange matrix is replaced by its negative. This uniform change
does not alter the classical exchange relations.

For $\overline w=(i_1,\ldots,i_r)$, Leclerc's displayed reduced
word is $(i_r,\ldots,i_1)$ for $w^{\mathrm L}$; its subscript
$j$ counts from the right-hand end. The reversal of our leftmost
subexpression for $v$ is his rightmost subexpression for
$v^{\mathrm L}$. His partial products therefore satisfy
\[
w^{\mathrm L}_{(j)}=s_{i_j}\cdots s_{i_1}=w_j^{-1},
\qquad
v^{\mathrm L}_{(j)}=\overline v_j^{-1}.
\]
Indeed, the recursion in \cite[Section~4.8.3]{leclerc2016cluster}
becomes the greedy recursion of Lemma~\ref{lem:vk} after taking
inverses. By our matrix-coefficient convention and
\cite[Section~2.2]{leclerc2016cluster}, we consequently have
\begin{equation}\label{eq:Ddelta}
\begin{aligned}
D\binom{w_j\varpi_{i_j}}{\overline v_j\varpi_{i_j}}\bigg|_{q=1}
&=\Delta_{\overline v_j\varpi_{i_j},\,w_j\varpi_{i_j}}\\
&=\Delta_{(v^{\mathrm L}_{(j)})^{-1}\varpi_{i_j},\,
          (w^{\mathrm L}_{(j)})^{-1}\varpi_{i_j}}.
\end{aligned}
\end{equation}
The last expression is precisely the minor in Leclerc's
Corollary~4.4. The inverses in that expression refer to his partial
products, not to ours. For instance, in the coordinates of
Example~\ref{ex:Tsystem-A2},
$D\binom{s_1s_2\varpi_2}{\varpi_2}|_{q=1}=xy-z$, whereas
$\Delta_{\varpi_2,(s_1s_2)^{-1}\varpi_2}=y$.

We will use the following factorization convention throughout this
subsection. Let
\[
R_w:=K_{\mathrm{cl}}(\mathscr C_w)\simeq\mathbb C[N(w)]
\]
be realized inside $\mathbb C[N]$ through the specialized ambient
cluster variables. In the transported Leclerc convention this is
$\mathbb C[N]^{N^{\prime}(w^{\mathrm L})}$, by
\cite[Section~4.5]{leclerc2016cluster}. The product decomposition
of \cite[Proposition~2.1(a)]{leclerc2016cluster} identifies
$\mathbb C[N]$ with a polynomial extension of $R_w$.
Both rings are factorial and have only nonzero scalar units;
in particular, a prime of $R_w$ remains prime in $\mathbb C[N]$.
Every mutable variable of an ambient seed is irreducible in
$R_w$ by \cite[Theorem~3.1]{geiss2013factorial}. The ambient frozen
variables are unchanged under mutation, and their irreducibility
follows from \cite[Corollary~4.4]{leclerc2016cluster} applied with
$v^{\mathrm L}=e$. Thus all ambient seed variables are prime in
both rings. All subsequent factorizations are taken before
localization; frozen variables are not units at this stage.

\begin{theorem}\label{thm:menard}
In finite ADE type, Leclerc's seed of
$\mathbb C[\mathring{\mathcal B}_{v,w}]$ coincides with M\'enard's
seed $\mathbf s(\overline v,\overline w)$, after adopting the
common global quiver convention of Section~\ref{sec:wordseed}.
\end{theorem}

\begin{proof}
For $v=e$, both seeds are the standard initial seed, so assume
$\ell(v)>0$. Put $J=[r]\setminus\{p_1,\ldots,p_{\ell(v)}\}$ and
\[
P:=\prod_{l\in J}
\Delta_{\overline v_l\varpi_{i_l},\,w_l\varpi_{i_l}}.
\]
By
\cite[Corollary~4.4]{leclerc2016cluster}, the distinct irreducible
factors of $P$ in $\mathbb C[N]$, up to nonzero scalar multiples,
are precisely the $\ell(w)-\ell(v)$ variables of Leclerc's
initial cluster, including its frozen variables.

Let \(\{x_l\mid l\in J\}\) be the retained variables of
\(\mathbf s(\overline v,\overline w)\), and write
\[
    \{y_1,\ldots,y_m\},
    \qquad
    m:=\ell(w)-\ell(v)=|J|,
\]
for the variables of Leclerc's initial cluster. Each \(x_l\) is
the specialization of a variable in an ambient seed of
\(K(\mathscr C_w)\), and hence is prime in \(\mathbb C[N]\) by the
preceding discussion. Moreover,
Theorem~\ref{thm:dividminors} and \eqref{eq:Ddelta} imply that
\(x_l\) divides \(P\).

By \cite[Corollary~4.4]{leclerc2016cluster}, the elements
\(y_1,\ldots,y_m\), considered up to multiplication by nonzero
scalars, are precisely the distinct irreducible factors of \(P\).
It follows that, for every \(l\in J\), there are a unique index
\(\sigma(l)\in[m]\) and a unique scalar
\(c_l\in\mathbb C^\times\) such that
\begin{equation}\label{eq:Menard-Leclerc-up-to-scalar}
    x_l=c_l y_{\sigma(l)}.
\end{equation}
The variables \(x_l\) are pairwise nonassociate, since they belong
to a single cluster and are algebraically independent. Hence
\(\sigma\) is injective. Since both variable sets have \(m\)
elements, it is a bijection
\[
    \sigma\colon J\xrightarrow{\sim}[m].
\]

We now show that every scalar \(c_l\) in
\eqref{eq:Menard-Leclerc-up-to-scalar} is equal to \(1\). Let
\(\varphi_M\in\mathbb C[N]\) denote the cluster character of a
preprojective module \(M\). Under the fixed realization inside
\(\mathbb C[N]\), every specialized ambient cluster variable is
equal, with its prescribed normalization, to the cluster character
of the corresponding reachable rigid module.

Indeed, for the standard ambient initial seed, this is the exact
generalized-minor identity of
\cite[Equation~(29)]{leclerc2016cluster}. Suppose that, at an
ambient seed associated with
\[
    T=\bigoplus_i T_i,
\]
the cluster variables are exactly \(x_i=\varphi_{T_i}\). If the seed
is mutated at \(k\), let \(T_k^*\) be the mutated summand and let
\(E_k,E_k'\) be the middle terms of the corresponding exchange
sequences. The cluster-character multiplication formula gives
\[
    \varphi_{T_k}\varphi_{T_k^*}
    =
    \varphi_{E_k}+\varphi_{E_k'}.
\]
Since cluster characters are multiplicative on direct sums, the two
terms on the right are precisely the two monomials in the classical
cluster exchange relation. Thus this identity is
\[
    x_k\varphi_{T_k^*}
    =
    \prod_i x_i^{[b_{ik}]_+}
    +
    \prod_i x_i^{[-b_{ik}]_+}.
\]
Because \(\mathbb C[N]\) is a domain and
\(x_k=\varphi_{T_k}\neq0\), the mutated variable is exactly
\(\varphi_{T_k^*}\). Induction along an ambient mutation sequence
therefore shows that, for every \(l\in J\),
\[
    x_l=\varphi_{M_l}
\]
for some reachable indecomposable rigid module \(M_l\). This is an
exact equality of regular functions, rather than an equality up to
a scalar.

Likewise, Leclerc's seed is defined from a basic maximal rigid
object
\[
    T_{\mathrm{Lec}}
    =\bigoplus_{a=1}^m L_a,
\]
and its variables are exactly
\[
    y_a=\varphi_{L_a}.
\]
By \cite[Section~4.6]{leclerc2016cluster}, if \(M\) is rigid, then
\(\varphi_M\) is an element of the dual semicanonical basis
\(\mathcal S^*\) of \(\mathbb C[N]\). Consequently,
\[
    x_l=\varphi_{M_l},
    \qquad
    y_{\sigma(l)}=\varphi_{L_{\sigma(l)}}
\]
are elements of the same basis \(\mathcal S^*\).

If these were distinct basis elements, then
\eqref{eq:Menard-Leclerc-up-to-scalar} would give the nontrivial
linear relation
\[
    x_l-c_l y_{\sigma(l)}=0
\]
between two distinct elements of \(\mathcal S^*\), contradicting
the linear independence of this basis. Hence
\[
    x_l=y_{\sigma(l)}.
\]
Substituting this equality into
\eqref{eq:Menard-Leclerc-up-to-scalar}, and using
\(y_{\sigma(l)}\neq0\), gives
\[
    c_l=1.
\]
Thus the two variable sets agree exactly, with their prescribed
normalizations.

It remains to compare the remaining seed data. By
\cite[Theorem~6.8 and Propositions~7.25--7.26]{menard2022cluster},
the retained ambient summands form the basic maximal rigid object
\[
    T_{\mathrm{Men}}
    =\bigoplus_{l\in J}M_l
\]
of \(\mathcal C^{\mathrm{pre}}_{v,w}\); its endomorphism quiver and
projective-injective summands yield M\'enard's exchange matrix and
frozen set.

The equality of normalized variables proved above gives
\[
    \varphi_{M_l}
    =
    \varphi_{L_{\sigma(l)}}
    \qquad(l\in J).
\]
A rigid module has an open orbit in the irreducible component whose
dual semicanonical basis element is its cluster character. Hence
\(M_l\) and \(L_{\sigma(l)}\) have open orbits in the same
irreducible component. Any two nonempty open subsets of an
irreducible variety intersect, whereas two distinct orbits are
disjoint. Their orbits must therefore coincide, and hence
\[
    M_l\simeq L_{\sigma(l)}
    \qquad(l\in J).
\]
It follows that
\[
    T_{\mathrm{Men}}\simeq T_{\mathrm{Lec}}.
\]
Their endomorphism quivers and projective-injective summands
therefore coincide. This identifies the exchange matrices and
frozen sets, up to the prescribed relabeling
\(\Phi^{\ell(v)}\).
\end{proof}

Next we identify the localized quantum Grothendieck ring. The proof
uses positive Laurent expansions in the ambient category
$\mathscr C_w$ and the classical Richardson realization. We first
recall the equality of the lower and upper quantum cluster algebras.

\begin{lemma}\label{lem:uppercluster}
Assume finite ADE type. For any $v\leq w$, with frozen variables
inverted and coefficients in $\KK=\mathbb Z[q^{\pm1/2}]$, one has
\[
\cA_q(\overline v,\overline w)=U_q(\overline v,\overline w).
\]
\end{lemma}

\begin{proof}
By \cite[Theorem~10.1]{casals2025cluster}, after the
flag-variety identification and the change of word indexing used
above, M\'enard's cluster structure coincides with the cluster
structure associated with a left-inductive weave
\(\mathfrak W\) for the corresponding braid variety. Hence, after
relabeling and adopting the common global quiver convention of
Section~\ref{sec:wordseed}, the exchange type of
\(\mathbf s(\overline v,\overline w)\) is represented by the
weave quiver \(Q_{\mathfrak W}\).

By \cite[Corollary~7.16]{casals2025cluster}, every quantum cluster
algebra whose exchange type is represented by
\(Q_{\mathfrak W}\) is equal to its corresponding quantum upper
cluster algebra. This result applies to any compatible quantum
seed of this exchange type: it depends only on the mutation class
of the mutable principal part of the exchange matrix and does not
require a prescribed compatible skew-symmetric form or
compatibility diagonal. Its coefficient convention is
\(\mathbb Z[q^{\pm1/2}]\), with the frozen variables inverted,
which is precisely the convention used here. Therefore
\[
    \cA_q(\overline v,\overline w)
    =
    U_q(\overline v,\overline w).
\]
\end{proof}

\begin{lemma}[Positive Laurent support under specialization]
\label{lem:positive-specialization}
Let $\widehat{\mathbf t}$ be a quantum seed of the ambient
quantum unipotent algebra $K(\mathscr C_w)$, with vertex set
$\widehat J$, and let $J\subseteq\widehat J$.
Write $X_k$ for its normalized quantum variables and
$x_k:=X_k|_{q^{1/2}=1}$ for the corresponding classical variables.
For a self-dual simple module $L\in\mathscr C_w$, put
$z_L:=\Phi([L])$. If, in the classical ambient function field,
\[
z_L\big|_{q^{1/2}=1}\in
\mathbb C[x_j^{\pm1}\mid j\in J],
\]
then $z_L$ belongs to the quantum subtorus generated by
$\{X_j^{\pm1}\mid j\in J\}$ over $\KK$.
\end{lemma}

\begin{proof}
The ambient positive Laurent phenomenon
\cite[Lemma~2.3]{kashiwara2019laurent} gives a finite expansion
\begin{equation}\label{eq:ambient-positive-Laurent}
z_L=\sum_{\mathbf a\in\mathbb Z^{\widehat J}}
c_{\mathbf a}(q^{1/2})X^{\mathbf a},
\qquad
c_{\mathbf a}(q^{1/2})\in\mathbb Z_{\geq0}[q^{\pm1/2}],
\end{equation}
where $X^{\mathbf a}$ denotes the normalized toric monomial.
For completeness, positivity follows by clearing one ambient
seed-monomial denominator. The resulting left-hand side is, up to
a power of $q^{1/2}$, the class of a convolution module, while the
right-hand monomials are powers of $q^{1/2}$ times pairwise distinct
real simple classes. Their coefficients are therefore graded
Jordan--H\"older multiplicities. Toric normalization changes these
coefficients only by powers of $q^{1/2}$.

Specialization of \eqref{eq:ambient-positive-Laurent} gives
\[
z_L\big|_{q^{1/2}=1}
=\sum_{\mathbf a}c_{\mathbf a}(1)x^{\mathbf a}.
\]
This is an identity in the classical ambient Laurent ring; its
compatibility with the Grothendieck-ring specialization can also
be checked after clearing the same monomial denominator.
The full family $(x_k)_{k\in\widehat J}$ is algebraically independent.
The hypothesis and uniqueness of Laurent expansion consequently imply
\[
c_{\mathbf a}(1)=0
\qquad\text{if }\mathbf a|_{\widehat J\setminus J}\neq0.
\]
A Laurent polynomial with nonnegative integer coefficients vanishes
at $1$ only if it is zero. Thus all such $c_{\mathbf a}$ vanish
identically, proving the assertion.
\end{proof}

\begin{theorem}\label{thm:categorification}
Assume finite ADE type. Let $v\leq w\in W$, and let
$\overline w$ be a reduced expression of $w$. Then, with the
quantum-seed normalization and the common coefficient ring,
\[
K(\widetilde{\mathscr{C}}_{w,v})=\cA_q(\overline{v},\overline{w}).
\]
In particular, \(\widetilde{\mathscr{C}}_{w,v}\) is a monoidal categorification of the quantum cluster algebra \(\cA_q(\overline{v},\overline{w})\).
\end{theorem}

\begin{proof}
We work in the fixed ambient skew field of $K(\mathscr C_w)$.
The case $v=e$ is the usual localized monoidal categorification,
so assume $\ell(v)>0$.

\medskip
\noindent
\textit{Step 1: the classical realization in the ambient field.}
Put $R_{w,v}:=K_{\mathrm{cl}}(\mathscr C_{w,v})\subset R_w$.
By \cite[Theorem~2.20(ii)(c) and Remarks~2.19, 2.21(ii)]{kashiwara2018monoidal},
the fixed matrix-coefficient realization identifies $R_{w,v}$ with
the following actual subring of $\mathbb C[N]$:
\[
\begin{gathered}
H_v:=N\cap vN^-v^{-1},\qquad
H'_w:=N\cap wNw^{-1},\\
R_{w,v}
=\{f\in\mathbb C[N]\mid
f(a n b)=f(n)\text{ for all }a\in H_v,\ n\in N,\ b\in H'_w\}.
\end{gathered}
\]
These groups are precisely $N(v^{\mathrm L})$ and
$N^{\prime}(w^{\mathrm L})$ in the transported Leclerc convention.
Set
\[
D_j:=M\binom{w\varpi_j}{v\varpi_j},\qquad
d_j:=\Phi([D_j])\big|_{q^{1/2}=1}
=\Delta_{v\varpi_j,w\varpi_j}\qquad(j\in I).
\]
Leclerc's localization theorem
\cite[Theorem~2.12]{leclerc2016cluster} realizes the classical
Richardson coordinate ring as
\[
R_{w,v}[d_j^{-1}\mid j\in I]
\subset\operatorname{Frac}(R_w).
\]
By Theorem~\ref{thm:menard}, its initial cluster variables are exactly
the specializations of our retained ambient variables. Consequently,
Theorem~\ref{thm:classical-upper-richardson} identifies this subring
with $U(\overline v,\overline w)$ in the same function field.
The identification continues to hold after every retained mutation,
since both sides use the same initial functions and exchange relations.
In particular, for every retained seed $\mathbf t$,
\begin{equation}\label{eq:classical-retained-Laurent}
R_{w,v}\subset
\mathbb C[x_j(\mathbf t)^{\pm1}\mid j\in J]
\subset\operatorname{Frac}(R_w),
\end{equation}
where $J$ is the retained vertex set. This uses a localization of
an invariant subring, not a specialization of deleted variables.

\medskip
\noindent
\textit{Step 2: the quantum upper inclusion.}
Let $\mathbf t$ be any seed obtained from
$\mathbf s(\overline v,\overline w)$ by retained mutations.
The zero-row argument in the proof of Theorem~\ref{thm:clustercat}
lifts this sequence to an ambient monoidal seed
$\widehat{\mathbf t}$ of $\mathscr C_w$, whose vertex set is
$\widehat J=J\sqcup E$, with $E$ the deleted vertex set.
The restricted toric frame on $J$ is precisely that of $\mathbf t$.
For a self-dual simple $L\in\mathscr C_{w,v}$,
\eqref{eq:classical-retained-Laurent} and
Lemma~\ref{lem:positive-specialization} give
\[
\Phi([L])\in\mathcal T_q(\mathbf t).
\]
The self-dual simple classes span $K(\mathscr C_{w,v})$ over
$\KK$, and $\mathbf t$ was arbitrary. Intersecting these
quantum tori inside the fixed ambient skew field therefore yields
\begin{equation}\label{eq:subsetKw-cat}
K(\mathscr C_{w,v})
\subset U_q(\overline v,\overline w).
\end{equation}

\medskip
\noindent
\textit{Step 3: comparison of localizations.}
Let \(J_{\mathrm{fr}}\subset J\) be the frozen vertex set, and let
\(X_a\) be the frozen simple module corresponding to
\(a\in J_{\mathrm{fr}}\). Set
\[
    x_a:=\Phi([X_a])\big|_{q^{1/2}=1}.
\]
By Theorem~\ref{thm:menard} and
\cite[Proposition~4.1 and Corollary~4.4]{leclerc2016cluster},
there are integers \(m_{aj}\geq0\) such that
\[
    d_j=\prod_{a\in J_{\mathrm{fr}}}x_a^{m_{aj}},
\]
and every frozen variable occurs in at least one of these
factorizations. The coefficient is \(1\), since \(d_j\) is the
cluster character of the corresponding projective-injective
object and cluster characters are multiplicative on direct sums.

Let \(L_j\) be the self-dual normalization of the convolution
monomial
\[
    \mathop{\bigcirc}_{a\in J_{\mathrm{fr}}}
    X_a^{\circ m_{aj}}.
\]
Because the frozen modules strongly commute, \(L_j\) is simple.
Both \(L_j\) and \(D_j\) specialize to \(d_j\). Since the
specialized classes of self-dual simple modules form a basis, it
follows that
\[
    D_j\simeq L_j.
\]

Consequently, localization at the \(D_j\)'s is the same as
localization at all frozen modules. Indeed, every \(D_j\) is a
frozen convolution monomial. Conversely, if \(X_a\) occurs in
\(D_j\), write, up to a grading shift,
\[
    D_j\simeq X_a\circ L
\]
for another frozen convolution monomial \(L\). After \(D_j\) is
inverted, the appropriate grading shift of \(L\circ D_j^{-1}\)
is an inverse of \(X_a\).

Therefore, by \eqref{eq:subsetKw-cat} and
Theorems~\ref{thm:inCwv} and~\ref{thm:clustercat},
\[
    \mathcal A_q(\overline v,\overline w)
    \subset K(\widetilde{\mathscr C}_{w,v})
    \subset U_q(\overline v,\overline w).
\]
Here the upper inclusion extends through the localization because
the \(D_j\)'s are frozen monomials and hence are invertible in
every quantum torus defining \(U_q\). By
Lemma~\ref{lem:uppercluster}, both inclusions are equalities. The
realization of normalized cluster monomials and exchange sequences
from Theorem~\ref{thm:clustercat} is preserved under localization,
which proves the monoidal categorification statement.
\end{proof}

\begin{remark}
The proof uses the classical realization only to determine the
support of an ambient Laurent expansion. Positivity is essential:
without it, a nonzero coefficient such as $q^{1/2}-1$ could disappear
at $q^{1/2}=1$. We neither assume that specialization is injective
nor form an ordinary quantum quotient by equations $X_e=1$.
In particular, the proof is independent of any Laurent-family
construction for $\mathscr C_{w,v}$.
\end{remark}


\begin{thebibliography}{CGGLSS25}

\bibitem[Bi24]{bi2024richardson}
Yingjin Bi,
\emph{Monoidal categorification on open Richardson varieties},
\href{https://arxiv.org/abs/2409.04715v4}{arXiv:2409.04715v4},
2024 (revised 2025).

\bibitem[BY25]{bao2025upper}
Huanchen Bao and Jeff York Ye,
\emph{Upper cluster structure on Kac--Moody Richardson varieties},
arXiv:2506.10382, 2025.

\bibitem[CGGLSS25]{casals2025cluster}
Roger Casals, Eugene Gorsky, Mikhail Gorsky, Ian Le, Linhui Shen, and Jos\'e Simental,
\emph{Cluster structures on braid varieties},
J. Amer. Math. Soc. 38 (2025), no.~2, 369--479.

\bibitem[GLSB26]{galashin2025braid}
Pavel Galashin, Thomas Lam, and Melissa Sherman-Bennett,
\emph{Braid variety cluster structures, II: General type},
Invent. Math. 243 (2026), no.~3, 1079--1127.

\bibitem[GLS13]{geiss2013cluster}
Christof Gei{\ss}, Bernard Leclerc, and Jan Schr\"oer,
\emph{Cluster structures on quantum coordinate rings},
Selecta Math. (N.S.) 19 (2013), no.~2, 337--397.

\bibitem[GLS13b]{geiss2013factorial}
Christof Gei{\ss}, Bernard Leclerc, and Jan Schr\"oer,
\emph{Factorial cluster algebras},
Doc. Math. 18 (2013), 249--274.

\bibitem[Ing19]{ingermanson2019cluster}
Grace Ingermanson,
\emph{Cluster algebras of open Richardson varieties},
Dissertation, 2019.

\bibitem[KKK18]{kang2018symmetric}
Seok-Jin Kang, Masaki Kashiwara, and Myungho Kim,
\emph{Symmetric quiver Hecke algebras and $R$-matrices of quantum affine algebras},
Invent. Math. 211 (2018), no.~2, 591--685.

\bibitem[KKKO18]{kang2018monoidal}
Seok-Jin Kang, Masaki Kashiwara, Myungho Kim, and Se-jin Oh,
\emph{Monoidal categorification of cluster algebras},
J. Amer. Math. Soc. 31 (2018), no.~2, 349--426;
\href{https://arxiv.org/abs/1801.05145}{arXiv:1801.05145} (merged version).

\bibitem[Kas93]{kashiwara1993global}
Masaki Kashiwara,
\emph{Global crystal bases of quantum groups},
Duke Math. J. 69 (1993), 455--485.

\bibitem[KK19]{kashiwara2019laurent}
Masaki Kashiwara and Myungho Kim,
\emph{Laurent phenomenon and simple modules of quiver Hecke algebras},
Compos. Math. 155 (2019), no.~12, 2263--2295.

\bibitem[KKOP18]{kashiwara2018monoidal}
Masaki Kashiwara, Myungho Kim, Se-jin Oh, and Euiyong Park,
\emph{Monoidal categories associated with strata of flag manifolds},
Adv. Math. 328 (2018), 959--1009.

\bibitem[KKOP23]{kashiwara2023localizations}
Masaki Kashiwara, Myungho Kim, Se-jin Oh, and Euiyong Park,
\emph{Localizations for quiver Hecke algebras II},
Proc. Lond. Math. Soc. 127 (2023), no.~4, 1134--1184.

\bibitem[KKOP26]{kashiwara2026seeds}
Masaki Kashiwara, Myungho Kim, Se-jin Oh, and Euiyong Park,
\emph{Monoidal seeds of the categories $\mathcal C_{w,v}$ over quiver
Hecke algebras},
\href{https://arxiv.org/abs/2608.15020v2}{arXiv:2608.15020v2}, 2026.

\bibitem[KL09]{khovanov2009diagrammatic}
Mikhail Khovanov and Aaron D. Lauda,
\emph{A diagrammatic approach to categorification of quantum groups I},
Represent. Theory 13 (2009), no.~14, 309--347.

\bibitem[Lec16]{leclerc2016cluster}
Bernard Leclerc,
\emph{Cluster structures on strata of flag varieties},
Adv. Math. 300 (2016), 190--228.

\bibitem[Lus10]{lusztig2010introduction}
George Lusztig,
\emph{Introduction to quantum groups},
Modern Birkh\"auser Classics, Birkh\"auser/Springer, New York, 2010.

\bibitem[Men22]{menard2022cluster}
\'Etienne M\'enard,
\emph{Cluster algebras associated with open Richardson varieties: an algorithm to compute initial seeds},
arXiv:2201.10292, 2022.

\bibitem[Qin24]{qin2024analogs}
Fan Qin,
\emph{Analogs of the dual canonical bases for cluster algebras from Lie theory},
\href{https://arxiv.org/abs/2407.02480v4}{arXiv:2407.02480v4}.

\bibitem[SSB24]{serhiyenko2024leclerc}
Khrystyna Serhiyenko and Melissa Sherman-Bennett,
\emph{Leclerc's conjecture on a cluster structure for type $A$ Richardson varieties},
Adv. Math. 447 (2024), 109698.

\bibitem[Ste05]{stembridge2005tight}
John R. Stembridge,
\emph{Tight quotients and double quotients in the Bruhat order},
Electron. J. Combin. 11 (2005), no.~2, Research Paper~14, 41 pp.

\bibitem[TW16]{tingley2016mirkovic}
Peter Tingley and Ben Webster,
\emph{Mirkovi\'c--Vilonen polytopes and Khovanov--Lauda--Rouquier algebras},
Compos. Math. 152 (2016), no.~8, 1648--1696.

\end{thebibliography}

\end{document}